\documentclass[12pt]{amsart}

\usepackage{amssymb, amscd, txfonts}
\usepackage[all]{xy} 
\usepackage[pdftex]{graphicx}

\numberwithin{equation}{section}

\newtheorem{theorem}{Theorem}[section]

\newtheorem{lemma}[theorem]{Lemma}

\theoremstyle{definition}
\newtheorem{definition}[theorem]{Definition}

\newtheorem{condition}[theorem]{Condition}

\theoremstyle{remark}
\newtheorem{remark}[theorem]{Remark}
\newtheorem{claim}[theorem]{Claim}

\newcommand{\Z}{\mathbb{Z}}
\newcommand{\Q}{\mathbb{Q}}

\newcommand{\C}{\mathbb{C}}

\newcommand{\proj}{{\mathbb P}}

\newcommand{\plane}{{\mathbb P}^{2}}
\newcommand{\sextic}{|\mathcal{O}_{{\mathbb P}^{2}}(6)|}
\newcommand{\cubic}{|\mathcal{O}_{{\mathbb P}^{2}}(3)|}
\newcommand{\conic}{|\mathcal{O}_{{\mathbb P}^{2}}(2)|}
\newcommand{\lin}{|\mathcal{O}_{{\mathbb P}^{2}}(1)|}
\newcommand{\moduli}{\mathcal{M}_{r,a,\delta}}

\begin{document}

%%%%%%% Title %%%%%%%%%%%%%%%%%%%%%%%%
\title[]{Higher Chow cycles on some $K3$ surfaces with involution}
\author[]{Shouhei Ma}
\author[]{Ken Sato}
\thanks{Supported by KAKENHI 21H00971 and 20H00112} 
\address{Department~of~Mathematics, Tokyo~Institute~of~Technology, Tokyo 152-8551, Japan}
\email{ma.s.ae@m.titech.ac.jp}
\email{sato.k.da@m.titech.ac.jp}
\subjclass[2020]{}
%\dedicatory{}

\begin{abstract}
We construct, for each $3\leq r \leq 17$, an explicit family of higher Chow cycles of type $(2, 1)$ 
on a family of lattice-polarized $K3$ surfaces of generic Picard rank $r$, 
and prove that the indecomposable part of this cycle is non-torsion for very general members of the family. 
These are the first explicit examples of such families in middle Picard rank. 
Our construction is based on singular double plane model of $K3$ surfaces, 
and the proof of indecomposability is done by a degeneration method. 
\end{abstract}

\maketitle

%\setcounter{tocdepth}{1}
%\tableofcontents

\section{Introduction}\label{sec: intro}

Higher Chow groups ${\rm CH}^{k}(X, n)$ of a smooth complex projective variety $X$ 
were introduced by Bloch in his seminal paper \cite{Bl} 
as analogues of the classical Chow groups 
that give cycle-theoretic interpretation of building pieces of the higher $K$-groups of $X$. 
They play important roles in motive theory and arithmetic geometry, 
but at the same time also have touch with classical algebraic geometry. 
When $X$ is a surface, 
one of the first non-classical cases is ${\rm CH}^{2}(X, 1)$. 
Higher Chow cycles of this type have already exhibited new and rich geometric pictures. 
For example, they have Abel-Jacobi map to the weight $2$ Jacobian, 
which shows an intriguing analogy with the classical theory for divisors on curves. 
%which are at the same time somewhat parallel to the classical Abel-Jacobi theory for curves. 

Being relatively new objects, 
some of the works on such cycles have targeted on construction of nontrivial examples. 
More precisely, for a cycle $\xi$ in ${\rm CH}^{2}(X, 1)$, 
its \textit{indecomposable part} $\xi_{ind}$ is defined as the image of $\xi$ in the quotient group 
\begin{equation*}
{\rm CH}^{2}(X, 1)_{ind} = {\rm CH}^{2}(X, 1)/{\rm Pic}(X)\otimes_{{\Z}}{\C}^{\ast}. 
\end{equation*}
The cycle $\xi$ is said to be \textit{indecomposable} if $\xi_{ind}\ne 0$. 
The quest of indecomposable cycles has been done especially for $K3$ surfaces, 
starting with the work of M\"uller-Stach \cite{MS}, 
further developed in \cite{Co}, \cite{dAMS}, \cite{CL}, \cite{Ke}, \cite{CDKL}, 
%further developed by Collino, del Angel and M\"uller-Stach, 
%Chen, Lewis, Doran and Kerr \cite{Co}, \cite{dAMS}, \cite{CL}, \cite{Ke}, \cite{CDKL}, 
%further developed by Collino \cite{Co}, del Angel and M\"uller-Stach \cite{dAMS}, 
%Chen-Lewis \cite{CL}, Kerr \cite{Ke}, Chen-Doran-Kerr-Lewis \cite{CDKL}, 
and more recently \cite{Sas}, \cite{Sr}, \cite{Sa}. 
In particular, in \cite{CDKL}, Chen-Doran-Kerr-Lewis proved the existence of an indecomposable cycle 
for a very general member of many (non-explicit) lattice-polarized families of $K3$ surfaces. 

However, in spite of these progresses, it was pointed out in the introduction of \cite{CDKL} that 
\textit{explicit} examples of indecomposable cycles had been successfully constructed 
for only few lattice-polarized families. 
%Since then, still not so many works have been done as regards to explicit construction (\cite{Sas}, \cite{Sa}). 
In particular, except for the earlier works \cite{MS}, \cite{Co} where quartic surfaces were considered ($r\leq 2$), 
all later examples have generic Picard rank $r \geq 17$. 
%Examples of explicit families of cycles are desirable also from the viewpoint of connection with modular forms (\cite{Ma3}). 

The purpose of this paper is to construct plenty examples of explicit families 
of higher Chow cycles on $K3$ surfaces in middle Picard rank. 
More precisely, for each $3\leq r \leq 17$, 
we construct an explicit family of indecomposable cycles on 
a family of lattice-polarized $K3$ surfaces of generic Picard rank $r$. 
We will show in fact that they are non-torsion, which is stronger than nontriviality.  
Our construction is based on singular double plane model of $K3$ surfaces, 
and is related to some classical geometry. 
For example, our families include the famous family associated to six lines on ${\plane}$ (\cite{MSY}), 
for which our cycles are obtained by drawing the seventh line (\S \ref{ssec: r=16}). 
Our families also include some classical ones such as those associated to 
Coble curves (\S \ref{ssec: r=11}), Halphen curves, and del Pezzo surfaces of degree $\leq 7$ (\S \ref{ssec: 3r10}). 
In the del Pezzo case, construction of the cycles is related to lines on the del Pezzo surfaces. 

Let us be more specific. 
We consider pairs $(X, \iota)$ of a $K3$ surface $X$ and its non-symplectic involution $\iota$. 
The $K3$ surface $X$ is naturally lattice-polarized by the $\iota$-invariant part of $H^2(X, {\Z})$, 
which we denote by $H_{+}(X, \iota)$. 
There are exactly $75$ such families (\cite{Ni}) labelled by triplets $(r, a, \delta)$, 
where $r$ is the rank of $H_{+}(X, \iota)$ 
and $(a, \delta)$ are invariants which determine the discriminant form of $H_{+}(X, \iota)$. 
They satisfy $a\leq \min (r, 22-r)$ and $\delta\in \{ 0, 1\}$. 
Our method works for many cases $(r, a, \delta)$ with $3\leq r \leq 17$, 
but in order to keep the length of the paper reasonable, 
we select one extremal $(a, \delta)$ as a representative for each $r$: 
we take $a=\min (r, 22-r)$ and $\delta=1$.  

Thus, for each $3\leq r \leq 17$, we take a family 
$(\mathfrak{X}, \iota)\to \hat{U}$ of 
$K3$ surfaces with non-symplectic involution of this type $(r, a, 1)$ which dominates 
the moduli space (of dimension $20-r$), 
and construct an explicit family $\xi$ of higher Chow cycles of type $(2, 1)$ on $\mathfrak{X}\to \hat{U}$ 
in a systematic manner. 
Our main result is the following nontriviality (Theorem \ref{thm: nontrivial}). 

\begin{theorem}\label{thm: intro}
The indecomposable part $(\xi_{u})_{ind}\in {\rm CH}^2(\mathfrak{X}_u, 1)_{ind}$ 
of the higher Chow cycle $\xi_{u}$ is non-torsion for very general $u\in \hat{U}$. 
\end{theorem}

Here \textit{very general} points of $\hat{U}$ means points in the complement of countably many proper analytic subsets. 
The parameter space $\hat{U}$ is not the moduli space, 
but if we restrict the family to a suitable subvariety of $\hat{U}$, 
we can get a family over an etale cover of a Zariski open set of the moduli space. 

Our cycle family $\xi$ is constructed as follows. 
The space $\hat{U}$ parametrizes a certain type of singular plane sextics, 
and the fiber $X=\mathfrak{X}_{u}$ of $\mathfrak{X}\to \hat{U}$ over $u\in \hat{U}$ 
is the minimal resolution of the double cover 
$\bar{\pi}\colon \bar{X}\to {\plane}$ 
branched over the sextic. 
This sextic will have two specified nodes $q_1$, $q_2$. 
We draw the line $L=\overline{q_1 q_2}$ joining them. 
Then we get two $(-2)$-curves on $X$: 
 the exceptional curve $Z_1$ over $q_1$, and the strict transform $Z_0$ of $\bar{\pi}^{-1}(L)$. 
These two $(-2)$-curves intersect transversely at two points. 
Then a choice of suitable rational functions on $Z_0$, $Z_1$ defines a higher Chow cycle of type $(2, 1)$ on $X$ 
(see \S \ref{ssec: from sextic to cycle}). 
This construction can be done in family, and this produces $\xi$. 
We can also use the $(-2)$-curve over $q_2$ in place of $Z_1$. 
This shows that ${\rm CH}^{2}(X, 1)_{ind}$ has rank $\geq 2$: see Remark \ref{rmk: rank 2}. 

%If we use the $(-2)$-curve over $q_{2}$ instead of $q_1$, 
%we get a second family of higher Chow cycles, say $\xi^{\dag}$. 
%Then the two elements $(\xi_{u})_{ind}$ and $(\xi_{u}^{\dag})_{ind}$ of 
%${\rm CH}^{2}(X, 1)_{ind}$ are linearly independent for very general $u$, 
%so ${\rm CH}^{2}(X, 1)_{ind}$ has rank $\geq 2$ (Remark \ref{rmk: rank 2}). 
%This gives a strengthening of Theorem \ref{thm: intro}. 

The proof of Theorem \ref{thm: intro} uses a degeneration method, and proceeds inductively on $r$ as follows. 
We observe that the family in the case $r+1$ can be obtained as a degeneration of the family in the case $r$. 
In the starting case $r=18$, the assertion of Theorem \ref{thm: intro} is essentially proved in \cite{Sa} based on another model. 
What we do in this paper is to run the induction: 
we want to deduce nontriviality for the family in the step $r$ 
from that for the degenerated family in the step $r+1$. 

In order to realize this idea, we use a variant of normal functions. 
We define the Jacobian $J(X, \iota)$ of $(X, \iota)=(\mathfrak{X}_{u}, \iota_{u})$ 
as the generalized complex torus attached to the Hodge structure on 
the dual lattice of the $\iota$-anti-invariant part of $H^2(X, {\Z})$ (see \S \ref{ssec: Jacobian}). 
Then we have the anti-invariant Abel-Jacobi map 
\begin{equation*}
\nu_{-} : {\rm CH}^{2}(X, 1) \to J(X, \iota). 
\end{equation*}
If we apply $\nu_{-}$ to our cycle $\xi_{u}$ and vary $u$, 
we get a holomorphic section of the Jacobian fibration over $\hat{U}$. 
This is the \textit{anti-invariant normal function} of $\xi$. 
There are several technical benefits of introducing $J(X, \iota)$: 
\begin{enumerate}
\item stability under deformation and mild behavior under degeneration; 
\item coincidence with the transcendental Jacobian for very general $u$; 
\item detection of a part of the regulator of our cycle which depends only on the underlying (marked) sextic. 
\end{enumerate} 
By (2), the proof of Theorem \ref{thm: intro} is reduced to 
verifying non-torsionness of the anti-invariant normal function of $\xi$. 
This is what we actually prove inductively. 
Then (1) enables us to do a degeneration argument in a reasonable way (\S \ref{ssec: degeneration lemma}). 
The property (3) is also crucial at various points. 

Looking over the whole process, the structure of this paper is as follows. 
The construction of cycle families starts with 
the del Pezzo surface of degree $7$ and an $A_2$-configuration of lines on it (the case $r=3$), 
and proceeds to lower-dimensional families by degeneration. 
(We climb up and down the roof of the Nikulin mountain \eqref{eqn: Nikulin mountain}.) 
On the other hand, the proof of nontriviality proceeds in the reverse direction, 
starting from the case $r=18$ (\cite{Sa}) and spreading to higher-dimensional families. 
This degeneration method, realized via the anti-invariant normal functions, 
would be the main technical novelty of this paper. 
We expect that a similar mechanism works for other series of lattice-polarized $K3$ families as well, 
which will convince us further the plentyness of families of higher Chow cycles. 

The rest of this paper is organized as follows. 
In \S \ref{sec: K3} we recall the basic theory of $K3$ surfaces with non-symplectic involution. 
In \S \ref{sec: CH21} we recall the basic theory of higher Chow cycles of type $(2, 1)$ on $K3$ surfaces. 
In \S \ref{sec: CH21 sextic} we explain how to construct higher Chow cycles from plane sextics with two nodes, 
and also define the Jacobian $J(X, \iota)$. 

Construction of the cycle families and the proof of Theorem \ref{thm: intro} are done in 
\S \ref{sec: recipe} and \S \ref{sec: case-by-case} for all $r$. 
Although our construction and proof are done quite systematically, 
there still remain some parts of case-by-case nature that depend on $r$ in early stages of defining the families. 
This prevents us from carrying out the process for all $r$ in one time. 
It is, however, unreasonable to spend many pages for repeating similar process. 
Therefore we decided to take the following style of presentation: 
\begin{itemize}
\item We can (and do) separate the case-by-case part and the uniform part of the whole process. 
The case-by-case part takes place only in the first few steps of the construction of family. 
The remaining part of the construction, as well as the proof of non-torsionness, can be done uniformly.  
\item We carry out the case-by-case part for each $r$ in \S \ref{sec: case-by-case}. 
\item In \S \ref{sec: recipe}, after summarizing the case-by-case constructions in \S \ref{sec: case-by-case}, 
we carry out the uniform process in full detail. 
%At this point, we postpone the detail of the case-by-case constructions to \S \ref{sec: case-by-case}. 
\end{itemize} 
Thus, logically, %\S \ref{sec: case-by-case} should be read first, and 
\S \ref{sec: recipe} should be read after \S \ref{sec: case-by-case}: 
strictly speaking, after each subsection of \S \ref{sec: case-by-case} repeatedly. 
We reverse this order in our presentation,  
with the intention of clarifying the whole picture. 
We hope this could help the readers, rather than confusing them. 
See \S \ref{ssec: recipe summary} for a more detailed summary of the whole process.

\section{$K3$ surfaces with non-symplectic involution}\label{sec: K3}

In this section we recall the basic theory of $K3$ surfaces with non-symplectic involution developed by Nikulin \cite{Ni}. 
Throughout this paper, unless stated otherwise, 
a \textit{$K3$ surface} means a smooth complex projective $K3$ surface. 
A \textit{lattice} means a free $\Z$-module $L$ of finite rank 
equipped with a nondegenerate symmetric form $L\times L\to \Z$. 
The dual lattice of $L$ is denoted by $L^{\vee}$. 
This is canonically embedded in $L_{\Q}=L\otimes_{\Z} \Q$ by the pairing. 

\subsection{Classification theory}\label{ssec: K3 classify}

Let $X$ be a $K3$ surface. 
An involution $\iota\colon X\to X$ of $X$ is called \textit{non-symplectic} 
if it acts by $-1$ on $H^0(K_{X})$. 
Following \cite{Yo}, \cite{Ma1}, 
we call a pair $(X, \iota)$ of a $K3$ surface and its non-symplectic involution a \textit{2-elementary $K3$ surface}. 
The invariant and the anti-invariant lattices of $(X, \iota)$ are defined by 
\begin{equation*}
H_{\pm}(X, \iota) := \{ \: l\in H^2(X, {\Z}) \: | \: \iota^{\ast}l=\pm l \: \}. 
\end{equation*}
Then $H_{+}(X, \iota)$ is a Lorentzian lattice, 
$H_{-}(X, \iota)$ has signature $(2, \ast)$, 
and their discriminant groups 
$A_{H_{\pm}}=H_{\pm}(X, \iota)^{\vee}/H_{\pm}(X, \iota)$ 
are (isomorphic) 2-elementary abelian groups. 
%We have $A_{H_{+}}\simeq A_{H_{-}}$ naturally. 
Since $H^{2,0}(X)\subset H_-(X, \iota)_{{\C}}$, 
we have $H_+(X, \iota)\subset {\rm NS}(X)$ and $T(X)\subset H_-(X, \iota)$, 
where ${\rm NS}(X)$ is the N\'eron-Severi lattice and 
$T(X)= {\rm NS}(X)^{\perp}\cap H^2(X, \Z)$ is the transcendental lattice of $X$. 

By Nikulin \cite{Ni}, the deformation type of $(X, \iota)$ is determined by 
the isometry class of the invariant lattice $H_+(X, \iota)$, 
which in turn is determined by its \textit{main invariant} $(r, a, \delta)$, where 
\begin{itemize}
\item $r$ is the rank of $H_+(X, \iota)$, 
\item $a$ is the dimension of $A_{H_{+}}$ over $\mathbb{F}_{2}$, and 
\item $\delta \in \{ 0, 1 \}$ is the parity of the natural ${\Q}/2{\Z}$-valued quadratic form on $A_{H_{+}}$ 
(the so-called \textit{discriminant form}). 
%namely $\delta=0$ if the quadratic form takes values only in ${\Z}/2{\Z}$, and $\delta=1$ otherwise. 
\end{itemize}
The fixed locus $X^{\iota}$ of the involution $\iota$ is a disjoint union of smooth curves. 
%whose topological type is determined by $(r, a, \delta)$. 
More specifically (\cite{Ni}), when $(r, a, \delta)\ne (10, 10, 0), (10, 8, 0)$, 
we have $X^{\iota}=C \sqcup E_{1} \sqcup \cdots \sqcup E_{k}$ 
where $C$ is a smooth curve of genus $g$ and $E_{i}$ are $(-2)$-curves with 
\begin{equation*}\label{eqn: (g, k)}
g=11- (r+a)/2, \qquad k=(r-a)/2. 
\end{equation*}
By the classification of Nikulin \cite{Ni}, there are exactly 75 deformation types $(r, a, \delta)$. 
Among them, 59 have $\delta=1$ and $16$ have $\delta=0$. 
The range of $(r, a)$ is contained in the region 
\begin{equation}\label{eqn: Nikulin mountain}  
1\leq r \leq 20, \quad 0\leq a\leq r, \quad r+a\leq 22, 
\end{equation}
and satisfies $r\equiv a$ mod $2$. 
See \cite{Ni} for the precise distribution of $(r, a, \delta)$. 

We denote by ${\moduli}$ the moduli space of 2-elementary $K3$ surfaces of invariant $(r, a, \delta)$. 
This is an irreducible quasi-projective variety of dimension $20-r$, 
realized as a Zariski open set of an orthogonal modular variety attached to the anti-invariant lattices 
via the period mapping. 
We will not need a precise description of ${\moduli}$ (see \cite{Yo}). 
The only property of the moduli space we will use later is the fact that  
\begin{equation*}
H_{+}(X, \iota) = \textrm{NS}(X) 
\end{equation*}
for very general members $(X, \iota)$ of ${\moduli}$. 
%i.e., for $(X, \iota)$ in the complement of countably many divisors of ${\moduli}$. 

In this paper we will be interested in families of 2-elementary $K3$ surfaces in the following two series 
(the roof of the Nikulin mountain \eqref{eqn: Nikulin mountain}): 
\begin{itemize}
\item[(A)] $3\leq r\leq 11$, $r=a$ ($\Leftrightarrow k=0$) and $\delta=1$
\item[(B)]  $12\leq r \leq 18$, $r+a=22$ ($\Leftrightarrow g=0$) and $\delta=1$ for $r\leq 17$ but $\delta=0$ for $r=18$. 
\end{itemize}
Thus, for each $3\leq r \leq 18$, 
we have selected one $(r, a, \delta)$ with maximal $a$. 
The choice of the parity $\delta$ is in fact almost unique: 
if $r\ne 10, 18$, we have only $\delta=1$ for the above value of $a$. 
The series (A) have invariant lattice 
$H_{+}=\langle 2 \rangle \oplus \langle -2 \rangle^{\oplus r-1}$, 
and the series (B) with $r\leq 17$ have anti-invariant lattice 
$H_{-}=\langle 2 \rangle^{\oplus 2} \oplus \langle -2 \rangle^{\oplus 20-r}$. 

Families of 2-elementary $K3$ surfaces in these series have been classically studied  
from various points of view. %(classical algebraic geometry, modular forms, automorphism groups, $\cdots$). 
This richness is one reason for our choice of these invariants $(r, a, \delta)$. 
Another reason is the existence of certain isogeny between the moduli spaces with common $r$ (\cite{Ma1}). 
The moduli spaces ${\moduli}$ with maximal $a$ are initial with respect to this relation. 
Thus we have chosen an extremal invariant $(r, a, \delta)$ for each $3\leq r \leq 18$.

\subsection{Plane sextics}\label{ssec: double plane} 

One standard way to construct 2-elementary $K3$ surfaces 
is to take double covers of ${\plane}$ branched over singular sextic curves. 
In this subsection we recall this construction. 

Let $B\subset {\plane}$ be a plane sextic with at most simple singularities. %(in the sense of \cite{BHPV} \S II.8). 
(In our later examples, $B$ will have at most nodes and ordinary triple points; 
when $r\leq 16$, only nodes.) 
Let $\bar{\pi}\colon \bar{X}\to {\plane}$ be the double cover of ${\plane}$ branched over $B$. 
Then $\bar{X}$ is a singular $K3$ surface with at most ADE singularities 
over the singularities of $B$ (of the same name).  
The minimal resolution $X\to \bar{X}$ of these rational double points is a smooth $K3$ surface. 
The covering transformation of $\bar{X}\to {\plane}$ extends to a non-symplectic involution $\iota$ of $X$. 
In this way we obtain a 2-elementary $K3$ surface $(X, \iota)$ from $B$. 
%
%In what follows, for simplicity of exposition, 
%we assume that $B$ has at most ordinary double points (nodes) and ordinary triple points. 
%(This is sufficient for later sections.) 
%Then $\bar{X}$ has $A_{1}$-points over the nodes of $B$ and $D_{4}$-points over the triple points of $B$. 

The quotient $Y=X/\iota$ is a smooth rational surface. 
We have the commutative diagram 
\begin{equation*}
\xymatrix{
X \ar[d] \ar[r] & Y \ar[d] \\ 
\bar{X} \ar[r]^{\bar{\pi}} & {\plane}
}
\end{equation*}
where $Y\to {\plane}$ is an explicit blow-up supported on the singular points of $B$. 
(For example, each node is blown up once.) 
%the triple points of $B$, and the three infinitely near points over each triple point given by the branches of $B$. 
%(see \cite{AN}, \cite{Ma1}). 
%One fixed rational curve is produced over each triple point of $B$, and no fixed curve over the nodes of $B$. 
The inverse images of the exceptional curves of $Y\to {\plane}$ in $X$ 
are the exceptional curves of $X\to \bar{X}$. 
Pullback by the quotient map $X\to Y$ gives an isomorphism 
${\rm NS}(Y)_{{\Q}}\to H_+(X, \iota)_{{\Q}}$ of ${\Q}$-linear spaces. 
%which multiplies the quadratic form by $2$. 
This shows that $H_+(X, \iota)_{{\Q}}$ is generated by the pullback of $\mathcal{O}_{{\plane}}(1)$ 
and the exceptional curves of $X\to \bar{X}$. 

Let us notice the following property which will be used in \S \ref{ssec: degeneration lemma}. 

\begin{lemma}\label{lem: H- vs IH}
Let $H^2(\bar{X}, {\Q})_{prim}$ be the primitive part of $H^2(\bar{X}, {\Q})$, i.e., 
the orthogonal complement of $\bar{\pi}^{\ast}\mathcal{O}_{{\plane}}(1)$. 
Pullback by $X\to \bar{X}$ defines an isomorphism 
$H^2(\bar{X}, {\Q})_{prim}\to H_{-}(X, \iota)_{{\Q}}$. 
\end{lemma}
 
Here $H^2(\bar{X}, {\Q})$ is naturally isomorphic to the intersection cohomology $IH^2(\bar{X}, {\Q})$ 
as $\bar{X}$ has at most rational double points, 
and the intersection pairing on $H^2(\bar{X}, {\Q})$ is the same as the one on $IH^2(\bar{X}, {\Q})$. 

\begin{proof}
The pullback $H^2(\bar{X}, {\Q})\to H^2(X, {\Q})$ is injective 
and the orthogonal complement of its image is generated by the exceptional curves over the rational double points. 
Since $H_+(X, \iota)_{{\Q}}$ is generated by 
the pullback of $\mathcal{O}_{{\plane}}(1)$ and these exceptional curves, 
this implies our assertion. 
\end{proof}

\section{Higher Chow cycles and normal functions}\label{sec: CH21}

In this section we recall the basic theory of higher Chow groups ${\rm CH}^{2}(X, 1)$ 
specializing to the case when $X$ is a $K3$ surface.

\subsection{Higher Chow cycles}\label{ssec: CH21}

Let $X$ be a $K3$ surface. 
The higher Chow groups ${\rm CH}^{k}(X, n)$ of $X$ were defined by Bloch \cite{Bl}. 
In this paper we are mainly interested in the case $(k, n)=(2, 1)$. 
In this case, it is well-known that $\mathrm{CH}^2(X,1)$ can be described as 
the middle homology group of the following complex 
(see, e.g., \cite{MS2} Corollary 5.3): 
\begin{equation*}
K_2^{\mathrm{M}} (\C(X)) \xrightarrow{T} 
\displaystyle\bigoplus_{Z\in X^{(1)}}\C(Z)^\times 
\xrightarrow{\mathrm{div}}\displaystyle\bigoplus_{p\in X^{(2)}}\Z\cdot p.  
\end{equation*}
Here $X^{(r)}$ denotes the set of irreducible closed subvarieties of $X$ of codimension $r$, 
and $T$ is the tame symbol map from the Milnor $K_2$-group of the function field of $X$. 
By this expression, 
higher Chow cycles in $\mathrm{CH}^2(X,1)$ can be represented by formal sums
\begin{equation}\label{formalsum}
\sum_j (Z_j, f_j)\in \displaystyle\bigoplus_{Z\in X^{(1)}}\C(Z)^\times
\end{equation}
where $Z_j$ are irreducible curves on $X$ and 
$f_j\in \C(Z_j)^\times$ is a non-zero rational function on the normalization of $Z_j$ 
such that $\sum_j {\mathrm{div}}(f_j) = 0$ as a $0$-cycle on $X$.

\subsection{Normal functions}\label{ssec: normal function}

Let $H=(H_{\Z}, F^{\bullet})$ be a $\Z$-Hodge structure of weight 2. 
The generalized intermediate Jacobian of $H$ is defined as the quotient group 
\begin{equation}\label{eqn: Jacobian HS}
J(H) = \frac{H_{\C}}{H_\Z+F^2H_{\C}}. 
\end{equation}
This is a generalized complex torus, i.e., 
a quotient of a ${\C}$-linear space by a discrete lattice. 
This construction can be considered in family: 
let $S$ be a complex manifold and $\mathcal{H} = (\mathcal{H}_{\Z}, F^{\bullet})$ 
be a variation of $\Z$-Hodge structures of weight 2 over $S$. 
Then 
\begin{equation*}
{\mathcal J}(\mathcal H) = 
\frac{\mathcal H_{\Z}\otimes \mathcal{O}_{S}}{\mathcal H_{\Z}+F^2\mathcal{H}}
\end{equation*}
is a family of generalized complex tori over $S$. 
A holomorphic section of ${\mathcal J}(\mathcal H) \to S$ 
satisfying the horizontality condition is called a \textit{normal function}. 

Let $X$ be a $K3$ surface. 
We write 
\begin{equation*}\label{eqn: Jacobian X}
J(X) = J(H^2(X, \Z)) 
\end{equation*} 
and simply call it the \textit{Jacobian} of $X$. 
%Here the second equality follows from the unimodularity of $H^2(X, \Z)$. 
%Then $J(X)$ is naturally isomorphic to the Deligne cohomology $H^3_{\mathcal D}(X, \Z(2))$ 
%(up to multiplication by $(2\pi i)^{-2}$). 
%By the identification of $\mathrm{CH}^2(X,1)$ with the motivic cohomology, 
We have the \textit{regulator map} 
\begin{equation}\label{regulatormap}
\nu \; : \; \mathrm{CH}^2(X,1) \to J(X) 
\end{equation}
as the Abel-Jacobi map for $\mathrm{CH}^2(X,1)$ (cf.~\cite{MS}, \cite{CDKL} etc). 

A family of higher Chow cycles gives rise to a normal function as follows. 
Let $\pi\colon X\rightarrow S$ be a smooth family of $K3$ surfaces over a complex manifold $S$. 
Let $\mathcal{J}(X/S\!) = \mathcal{J}(R^2\pi_{\ast}\Z)$ be the family of Jacobians attached to $R^2\pi_{\ast}\Z$.  
Suppose that we have irreducible divisors $Z_j \subset X$ which are smooth over $S$ and 
nonzero meromorphic functions $f_j$ on $Z_j$ whose zeros and poles are also smooth over $S$ 
such that $\sum_{j}{\mathrm{div}}(f_j) = 0$. 
By restricting $Z_j$ and $f_j$ to each fiber of $\pi$, we obtain the family
\begin{equation*}
\xi = (\xi_{s})_{s\in S} = \left( \sum_j \left((Z_j)_s, (f_j)_s\right)\right)_{s\in S} 
\end{equation*}
of higher Chow cycles over $S$. 
Then we have the section $\nu(\xi)$ of $\mathcal{J}(X/S\!) \to S$ defined by 
\begin{equation*}
\nu(\xi)(s) = \nu(\xi_s), \quad s\in S.  
\end{equation*}
This section is holomorphic (\cite{CL} Proposition 4.1) 
and satisfies the horizontality condition (\cite{CDKL} Remark 2.1), 
namely it is a normal function. 
We call $\nu({\xi})$ the normal function of $\xi$. 
%In particular, the zero locus of $\nu(\xi)$ is an analytic subset on $S$. 

\subsection{Indecomposable part}\label{ssec: indecomposable}

Let $X$ be a $K3$ surface. 
By the intersection product of higher Chow cycles (\cite{Bl}), 
we have a group homomorphism
\begin{equation}\label{intersectionproduct}
\mathrm{Pic}(X)\otimes_\Z \Gamma(X,\mathcal{O}_X^\times) \longrightarrow \mathrm{CH}^2(X,1). 
\end{equation}
Here, of course, $\Gamma(X,\mathcal{O}_X^\times)={\C}^{\ast}$. 
If $Z$ is an irreducible curve on $X$ and $\alpha\in {\C}^{\ast}$, 
the image of $[Z]\otimes \alpha$ by \eqref{intersectionproduct} is represented by 
$(Z,\alpha)$ in the presentation \eqref{formalsum}. 
The cokernel of the map \eqref{intersectionproduct} is denoted by $\mathrm{CH}^2(X,1)_{ind}$. 
The image of a higher Chow cycle $\xi\in \mathrm{CH}^2(X,1)$ in $\mathrm{CH}^2(X,1)_{ind}$ is 
denoted by $\xi_{ind}$ and called the \textit{indecomposable part} of $\xi$. 
The cycle $\xi$ is said to be \textit{indecomposable} if $\xi_{ind}\ne 0$. 

A standard way to detect indecomposability is to use the transcendental regulator. 
Let $T(X)$ be the transcendental lattice of $X$ and $T(X)^{\vee}$ be its dual lattice. 
We regard $T(X)^{\vee}$ as an overlattice of $T(X)$ naturally. 
Then $T(X)^{\vee}$ is a ${\Z}$-Hodge structure of weight $2$. 
We denote  
\begin{equation*}
J_{tr}(X) = J(T(X)^\vee) 
\end{equation*}
and call it the \textit{transcendendal Jacobian} of $X$. 
By the unimodularity of $H^2(X, {\Z})$, 
the image of the orthogonal projection 
$H^2(X, {\Z}) \to T(X)_{{\Q}}$ is $T(X)^{\vee}$. 
This projection 
$H^2(X, {\Z}) \twoheadrightarrow T(X)^{\vee}$ 
is a morphism of ${\Z}$-Hodge structures and hence induces a surjective morphism 
$J(X)\to J_{tr}(X)$ 
of generalized complex tori with kernel $\mathrm{NS}(X)\otimes_{{\Z}} {\C}^{\ast}$. 
If $(Z, \alpha)$ is a decomposable cycle, its regulator is 
\begin{equation}\label{eqn: regulator decomposable cycle}
\nu(Z, \alpha) = [Z]\otimes \alpha \; \; \; \in \; \mathrm{NS}(X)\otimes_{{\Z}} {\C}^{\ast} \subset J(X). 
\end{equation}
Hence the composition of the regulator map \eqref{regulatormap} 
with the projection $J(X)\to J_{tr}(X)$ 
factors through $\mathrm{CH}^2(X,1)_{ind}$: %(cf.~\cite{dAMS} p.582): 
\begin{equation*}
\xymatrix{
\mathrm{CH}^2(X,1) \ar[r]^-{\nu} \ar@{->>}[d] & J(X) \ar@{->>}[d] \\
\mathrm{CH}^2(X,1)_{ind} \ar@{.>}[r]  & J_{tr}(X)
}
\end{equation*}
We call the induced map 
$\mathrm{CH}^2(X,1)_{ind} \rightarrow J_{tr}(X)$ 
the \textit{transcendental regulator map} and denote it by $\nu_{tr}$. 

\begin{remark}
Some authors use the composition map 
\begin{equation*}
\mathrm{CH}^2(X,1)_{ind} \rightarrow J_{tr}(X)\to H^{2,0}(X)^\vee/T(X)^{\vee} 
\end{equation*}
to detect indecomposability and call it the transcendental regulator. 
Since we will not use this composition but only the first map in this paper, 
we use this terminology for the first map, 
in view of the compatibility with ``transcendental lattice" and ``transcendental Jacobian". 
\end{remark}

\section{Higher Chow cycles attached to certain plane sextics}\label{sec: CH21 sextic}

In this section we explain the main construction in this paper, 
which associates to a plane sextic with two nodes 
a higher Chow cycle on the double covering $K3$ surface $(X, \iota)$. 
%Such a construction was first considered in \cite{Sa} for bidegree $(4, 4)$ curves on ${\proj}^1\times {\proj}^1$. 
Previously this type of constructions have been considered in \cite{Sr2}, \cite{Sr}, \cite{Sa} in some special cases. 
A technical novelty here is the introduction of the Jacobian $J(X, \iota)$ of $(X, \iota)$ (\S \ref{ssec: Jacobian}). 
This has several benefits as summarized in \S \ref{sec: intro}, which we will see in due course.

\subsection{Jacobian of $(X, \iota)$}\label{ssec: Jacobian}

Let $(X, \iota)$ be a 2-elementary $K3$ surface. 
As in \S \ref{sec: K3}, we denote by $H_{-}(X, \iota)$ the $\iota$-anti-invariant part of $H^2(X, {\Z})$ 
and let $H_{-}(X, \iota)^{\vee}$ be its dual lattice. 
Since $H^{2,0}(X)\subset H_{-}(X, \iota)_{{\C}}$, 
then $H_{-}(X, \iota)^{\vee}$ is equipped with a weight $2$ Hodge structure. 
We denote by 
\begin{equation*}\label{eqn: J(X,i)} 
J(X, \iota) = J( H_{-}(X, \iota)^{\vee})  
\end{equation*}
the generalized intermediate Jacobian of $H_{-}(X, \iota)^{\vee}$  
and call it the \textit{Jacobian} of $(X, \iota)$. 
As in \S \ref{ssec: indecomposable}, we have the orthogonal projections 
\begin{equation*}
H^2(X, {\Z}) =H^2(X, {\Z})^{\vee} \twoheadrightarrow H_{-}(X, \iota)^{\vee} \twoheadrightarrow T(X)^{\vee}, 
\end{equation*}
which induce surjective homomorphisms 
\begin{equation*}
J(X) \twoheadrightarrow J(X, \iota) \twoheadrightarrow J_{tr}(X) 
\end{equation*}
of generalized complex tori. 
The kernel of $J(X) \to J(X, \iota)$ is 
$H_{+}(X, \iota)\otimes_{{\Z}}{\C}^{\ast}$. 
%Since 
%$H^2(X, {\Z}) \to H_{-}(X, \iota)^{\vee}$ and $H^2(X, {\Z})\to T(X)^{\vee}$ 
%are surjective by the unimodularity of $H^2(X, {\Z})$, 
%so is $H_{-}(X, \iota)^{\vee} \to T(X)^{\vee}$. 
%It follows that the kernels of 
%$J(X) \to J(X, \iota)$ and $J(X, \iota) \to J_{tr}(X)$ 
%are connected. 
%More precisely, we have 
%\begin{equation*}
%\ker (J(X) \to J(X, \iota)) = H_{+}(X, \iota)\otimes_{{\Z}}{\C}^{\ast}, 
%\end{equation*} 
%\begin{equation*}
%\ker (J(X, \iota) \to J_{tr}(X)) = (H_{-}(X, \iota)^{\vee}\cap {\rm NS}(X)_{{\Q}})\otimes_{{\Z}}{\C}^{\ast}.  
%\end{equation*} 
When ${\rm NS}(X)=H_+(X, \iota)$, we have $J(X, \iota)=J_{tr}(X)$. 

The transcendental Jacobian $J_{tr}(X)$ is indispensable for detecting the indecomposable parts of higher Chow cycles 
(\S \ref{ssec: indecomposable}). 
However, since the Picard number is far from being stable under deformation, so is $J_{tr}(X)$. 
On the other hand, $J(X, \iota)$ behaves smoothly under deformation of $(X, \iota)$, 
while it coincides with $J_{tr}(X)$ for very general $(X, \iota)$. 
This is one benefit of considering $J(X, \iota)$. 
%Another benefit will be explained in Lemma \ref{lem: anti-inv regulator}. 

\subsection{Marking of plane sextics}\label{ssec: marking}

Let $B\subset {\plane}$ be a plane sextic with at most simple singularities. 
We assume that $B$ has at least two nodes. 
In order to construct a higher Chow cycle on the associated $K3$ surface, 
we need to attach a certain type of marking to $B$. 
We define it in two steps. 

\begin{definition}\label{def: weak marking}
We call a choice of two ordered nodes of $B$ a \textit{weak marking} of $B$. 
Often we denote these two nodes by $q_1$ and $q_2$. 
\end{definition}

Let $L=\overline{q_{1} q_{2}}$ be the line joining the marked nodes $q_1$, $q_{2}$. 
In what follows, we impose the following genericity assumption: 

\begin{condition}\label{condition: genericity}
The line $L$ intersects with $B$ at $q_1$, $q_2$ and two other smooth points of $B$. 
\end{condition}

\noindent
This condition implies in particular that 
$L$ intersects with $B$ transversely at the two smooth points  
and also transversely with each branch of $B$ at the nodes $q_1, q_2$. 

Let $\bar{\pi}\colon \bar{X}\to {\plane}$ be the double cover branched over $B$. 
As explained in \S \ref{ssec: double plane}, $\bar{X}$ has $A_1$-points over $q_1$ and $q_2$. 
The inverse image $\bar{\pi}^{-1}(L)\subset \bar{X}$ of $L$ is an irreducible rational curve, 
which has nodes at $\bar{\pi}^{-1}(q_1)$ and $\bar{\pi}^{-1}(q_2)$ and has no other singularity. 
At each node, the curve $\bar{\pi}^{-1}(L)$ has two branches. 

\begin{definition}\label{def: strong marking}
We call a bijective map from 
the reference set $\{ 0, \infty \}$ to the two branches of $\bar{\pi}^{-1}(L)$ at the node $\bar{\pi}^{-1}(q_1)$ 
(with the underlying weak marking) a \textit{strong marking} of $B$. 
This is equivalent to a choice of a branch of $\bar{\pi}^{-1}(L)$ at $\bar{\pi}^{-1}(q_1)$. 
Often we denote such a marking by $\mu$. 
\end{definition}

Geometrically, the two branches of $\bar{\pi}^{-1}(L)$ at $\bar{\pi}^{-1}(q_1)$ give 
two points in the projectivized tangent cone of $\bar{X}$ at $\bar{\pi}^{-1}(q_1)$ 
(which is the exceptional curve in the minimal resolution). 
%Hence a strong marking is the same as a bijective map from the reference set $\{ 0, \infty \}$ to these two points. 
For each weakly-marked sextic, 
there are two choices of promoting the weak marking to a strong marking.  
We call these two strong markings to be \textit{conjugate}.

\subsection{The main construction}\label{ssec: from sextic to cycle}

Let $(B, \mu)$ be a strongly-marked plane sextic. 
We keep the notations $\bar{X}$, $L$, $q_1$, $q_2$ in \S \ref{ssec: marking}. 
Let $X\to \bar{X}$ be the minimal resolution of the rational double points of $\bar{X}$ 
and $\iota$ be the associated non-symplectic involution of $X$. 
In this subsection we construct a higher Chow cycle on $X$ from $(B, \mu)$. 

Let $Z_0\subset X$ be the strict transform of $\bar{\pi}^{-1}(L)$ in $X$. 
Then $Z_0$ is a $(-2)$-curve and gives a normalization of $\bar{\pi}^{-1}(L)$. 
Let $Z_1\subset X$ be the $(-2)$-curve over the $A_1$-point $\bar{\pi}^{-1}(q_1)$. 
Then $Z_0$ intersects transversely with $Z_1$ at the two points 
corresponding to the two branches of $\bar{\pi}^{-1}(L)$ at $\bar{\pi}^{-1}(q_1)$. 
According to our strong marking, we denote these two points by $p_{0}$ and $p_{\infty}$. 
See Figure 1. 

\begin{figure}[h]\label{figure: main}
\includegraphics[height=32mm, width=80mm]{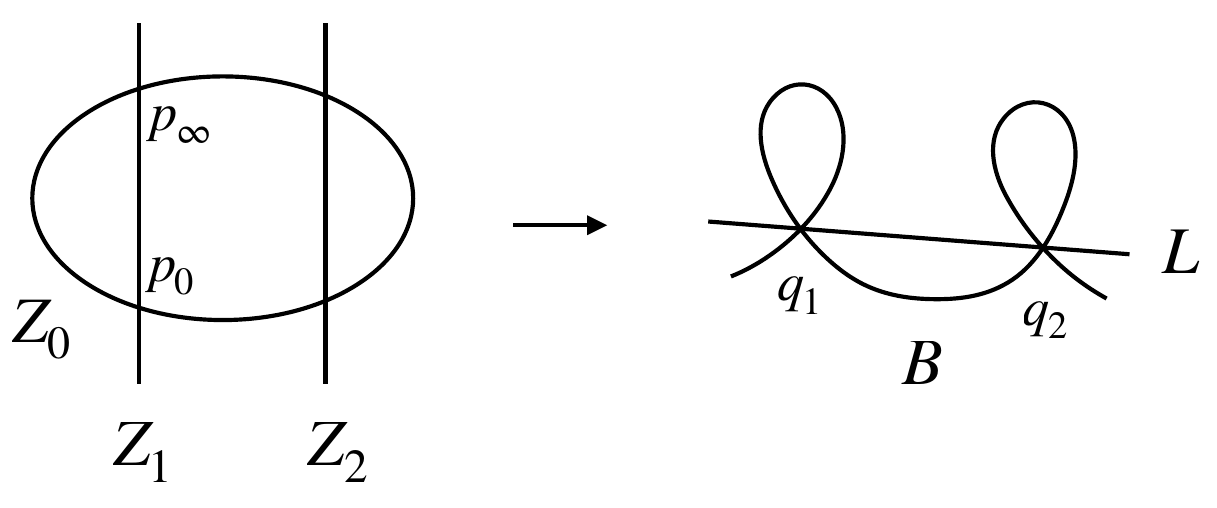}
\caption{}
\end{figure}

Now, since $Z_0$, $Z_1$ are smooth rational curves, 
we can choose a rational function $f_0$ on $Z_0$ and a rational function $f_1$ on $Z_1$ such that 
\begin{equation*}
{\rm div}(f_0) = (p_0) - (p_{\infty}), \qquad {\rm div}(f_1) = (p_{\infty}) - (p_0). 
\end{equation*}
Then, as explained in \S \ref{ssec: CH21}, the formal sum 
\begin{equation*}
\xi \: = \: (Z_0, f_0) + (Z_1, f_1) 
\end{equation*}
gives a higher Chow cycle of type $(2, 1)$ on $X$. 
We call $\xi$ a higher Chow cycle attached to the strongly-marked sextic $(B, \mu)$. 

It should be noted that we have freedom of choice of the rational functions $f_0$, $f_1$, 
but the ambiguity is only multiplication by some nonzero constants $\alpha_0, \alpha_1 \in {\C}^{\ast}$ 
on $f_0, f_1$ respectively. 
Its effect on the higher Chow cycle is to add the decomposable cycle 
\begin{equation}\label{eqn: ambiguity decomposable}
(Z_0, \alpha_0) + (Z_1, \alpha_1)
\end{equation}
to $\xi$. 
In particular, the indecomposable part of $\xi$ depends only on the marked sextic $(B, \mu)$. 
Moreover, we have the following. 

\begin{lemma}\label{lem: anti-inv regulator}
The image of the regulator $\nu(\xi)$ of $\xi$ by the projection $J(X)\to J(X, \iota)$ 
depends only on the marked sextic $(B, \mu)$. 
\end{lemma}

\begin{proof}
Since the supports $Z_0$, $Z_1$ of $(Z_0, \alpha_0)$, $(Z_1, \alpha_1)$ are curves preserved by $\iota$, 
their classes are contained in $H_{+}(Z, \iota)$. 
In view of \eqref{eqn: regulator decomposable cycle}, we see that 
the regulator of the decomposable cycle \eqref{eqn: ambiguity decomposable} 
is contained in the subgroup 
$H_{+}(Z, \iota)\otimes_{{\Z}}{\C}^{\ast}$ of $J(X)$. 
Therefore the image of $\nu(\xi)$ in $J(X, \iota)$ 
does not depend on the choice of $f_0, f_1$. 
\end{proof}

We denote by 
\begin{equation}\label{eqn: anti-invariant regulator}
\nu_{-}(\xi) = \nu_{-}(B, \mu) \: \: \: \in \: J(X, \iota) 
\end{equation}
the image of $\nu(\xi)$ by $J(X)\to J(X, \iota)$ and call it the \textit{anti-invariant regulator} of $\xi$. 
In this way, by considering $J(X, \iota)$, 
we could extract a part of the regulator of $\xi$ which depends only on $(B, \mu)$. 
This will be crucial in our later constructions.  

We also note the following. 

\begin{lemma}\label{lem: regulator conjugate}
Let $\mu^{\dag}$ be the strong marking of $B$ conjugate to $\mu$. 
Then we have 
$\nu_{-}(B, \mu^{\dag})=-\nu_{-}(B, \mu)$. 
In particular, $\nu_{-}(B, \mu)$ is non-torsion if and only if $\nu_{-}(B, \mu^{\dag})$ is so. 
\end{lemma}

\begin{proof}
For $\mu^{\dag}$, the labeling of $p_0$ and $p_{\infty}$ are exchanged. 
Hence 
$\xi^{\dag}  = (Z_0, f_0^{-1})+(Z_1, f_1^{-1})$ 
gives a higher Chow cycle attached to $(B, \mu^{\dag})$. 
Since $\xi^{\dag}=-\xi$ in ${\rm CH}^{2}(X, 1)$, 
we have $\nu(\xi^{\dag})=-\nu(\xi)$ and so $\nu_{-}(B, \mu^{\dag})=-\nu_{-}(B, \mu)$. 
\end{proof}

This shows that $\pm \nu_{-}(\xi)\in J(X, \iota)/\!-\!1$ depends only on the underlying 
\textit{weakly}-marked sextic.

\begin{remark}\label{rmk: elliptic fibration}
We can also see $\xi$ from the viewpoint of elliptic fibration. 
The divisor class of $Z_0+Z_1$ is an isotropic, nef and primitive vector of ${\rm NS}(X)$. 
Hence it is the class of a smooth elliptic curve and defines an elliptic fibration $X\to {\proj}^1$. 
Explicitly, this is given by the pullback of the pencil of lines on ${\plane}$ passing through $q_2$. 
Then $Z_0+Z_1$ is a singular fiber of type $I_2$ of this elliptic fibration. 
Thus the construction of $\xi$ can be seen as a standard one attached to an $I_n$-fiber of an elliptic surface. 
On the other hand, it is usually not easy to prove indecomposability of such a cycle (see the introduction of \cite{CDKL}). 
Our proof (\S \ref{sec: recipe}) does not make use of the elliptic fibration. 
\end{remark}

\section{The whole process}\label{sec: recipe}

Now we proceed to constructing our families of higher Chow cycles and proving their nontriviality. 
The whole process can be separated into the case-by-case part and the uniform part, 
and has an inductive structure. 
In this \S \ref{sec: recipe} we carry out the uniform part, 
postponing the case-by-case part to the next \S \ref{sec: case-by-case}. 
%The reason of adopting this style is explained in \S \ref{ssec: recipe summary}. 

\subsection{Summary of the whole process}\label{ssec: recipe summary}

Let $3\leq r \leq 18$. 
As announced in \S \ref{ssec: K3 classify}, 
for each such $r$, we set the invariant $a$ by $a=r$ when $r\leq 11$, and $a=22-r$ when $r\geq 11$. 
We also set the parity $\delta$ by $\delta=1$ when $r\leq 17$, and $\delta=0$ when $r=18$. 
Our goal in this paper is to 
\begin{enumerate}
\item construct a family $(X, \iota)\to \hat{U}$ of 2-elementary $K3$ surfaces with this invariant $(r, a, \delta)$ 
which dominates the moduli space ${\moduli}$, 
\item construct a family $\xi$ of higher Chow cycles on $X\to \hat{U}$, and 
\item prove that $(\xi_{u})_{ind}$ is non-torsion for very general $u\in \hat{U}$. 
\end{enumerate}

%The space $\tilde{U}$ will be taken as a parameter space of certain plane sextics with marking, 
%and $(X, \iota, \xi)$ will be constructed from this family of marked sextics 
%by the procedure of \label{ssec: from sextic to cycle}. 
%We have the inductive structure that 
%the parameter space $\tilde{U}_{r+1}$ for the next invariant $r+1$ 
%appears as a boundary of the parameter space $\tilde{U}=\tilde{U}_{r}$ for the invariant $r$. 
%Our proof of the nontriviality of $\xi$ proceeds by induction using this degeneration relation. 
%The starting case $r=18$ was essentially treated in the paper \cite{Sa} of second-named author. 

The construction of the family $(X, \iota, \xi)\to \hat{U}$ is divided into the part 
that requires case-by-case construction depending on $r$, 
and the part that can be done uniformly. 
%The case-by-case part comes first, and the uniform part comes next. 
After constructing the family, we prove its nontriviality by induction on $20-r$. 
The starting case $r=18$ was essentially done in \cite{Sa} (see \S \ref{ssec: r=18}). 
What we do here is to run the induction, and this argument is uniform. 
%Thus the whole process can be separated into the case-by-case part and the uniform part 
%and has an inductive structure. 

%The purpose of this \S \ref{sec: recipe} is to carry out the uniform part, %in full detail, 
%modulo the case-by-case construction that will be done in \S \ref{sec: case-by-case}.  
%Therefore, logically, this \S \ref{sec: recipe} should be read after 
%each subsection of \S \ref{sec: case-by-case} repeatedly: 
%\begin{equation*}
%\S \ref{ssec: r=18} \to \S \ref{sec: recipe} \to \S \ref{ssec: r=17} \to \S \ref{sec: recipe} \to \cdots. 
%\end{equation*} 
%However, in order to avoid repetition of similar argument and to clarify the main structure of this paper in advance, 
%we decided to present the uniform part first. 
%This would also facilitate to read \S \ref{sec: case-by-case}. 
In the logical order, the process proceeds as follows. 
We write $\hat{U}=\hat{U}_{r}$ when we want to specify the invariant $r$. 

\begin{enumerate}
\item In each subsection of \S \ref{sec: case-by-case}, we take a parameter space $\tilde{U}_{r}$ of 
certain plane sextics with weak marking. 
This is the case-by-case part, and is summarized in \S \ref{ssec: recipe case-by-case}. 
\item In \S \ref{ssec: recipe uniform}, we present a process by which 
we can systematically produce from $\tilde{U}_{r}$ 
a family $(X, \iota, \xi)\to \hat{U}_{r}$ of 2-elementary $K3$ surfaces and higher Chow cycles on them 
over a double cover $\hat{U}_{r}$ of $\tilde{U}_{r}$.  
\item We observe that the family over $\hat{U}_{r+1}$ is a degeneration of the family over $\hat{U}_{r}$.  
This is realized by constructing a partial compactification $\hat{U}_{r}'$ of $\hat{U}_{r}$, 
also in \S \ref{ssec: recipe uniform}. 
\item Suppose that we have proved nontriviality of our cycles over $\hat{U}_{r+1}$. 
%(In the starting case $r+1=18$, this was essentially proved in \cite{Sa} as we will explain in .) 
\item Then we prove nontriviality of our cycles over $\hat{U}_{r}$ 
by running the induction argument in \S \ref{ssec: nontriviality}. 
In fact, what we actually prove inductively is non-torsionness of the anti-invariant normal function of our cycles 
(Theorem \ref{thm: nontrivial normal function}). 
This passage is crucial in our proof. 
\item Repeat the process $(4) \to (5)$ until $r=3$ where our construction stops. 
\end{enumerate}

%Introduction of the anti-invariant normal functions is the main technical point for running our induction systematically. 

In the rest of \S \ref{sec: recipe}, we carry out the uniform process (2) -- (6) in full detail, 
modulo the ``input'' step (1) which will be worked out in \S \ref{sec: case-by-case}. 
%Thus, logically, this should be read after 
%each subsection of \S \ref{sec: case-by-case} repeatedly. 
We adopt this style of presentation for avoiding repetition of similar argument and 
for clarifying the picture of the whole process. 
%In \S \ref{ssec: recipe case-by-case} we summarize the case-by-case part of the construction. 
%In \S \ref{ssec: recipe uniform} we present the uniform part of the construction. 
%Thus the families $(X, \iota, \xi)\to \hat{U}_{r}$ will get produced at the end of \S \ref{ssec: recipe uniform}. 
%The inductive proof of non-torsionness is done in \S \ref{ssec: nontriviality}, 
%using a lemma prepared in \S \ref{ssec: degeneration lemma}. 

\subsection{Case-by-case part of the construction}\label{ssec: recipe case-by-case}

In \S \ref{sec: case-by-case}, we will 
define for each $r$ a space $\tilde{U}_{r}$ which parametrizes certain marked plane sextics, 
together with its partial compactification $\tilde{U}_{r}'$. 
%We proceed in two steps: 
%first we consider sextics without marking, and then endow them with certain markings. 
These are the inputs for the whole process. 
In this \S \ref{ssec: recipe case-by-case}, we summarize a common story of 
the constructions in \S \ref{sec: case-by-case} in some detail so that 
we can proceed to the next \S \ref{ssec: recipe uniform}. 
The family of cycles will get produced at the end of \S \ref{ssec: recipe uniform}. 
%Thus the parameter spaces $\tilde{U}_{r}\subset \tilde{U}_{r}'$ 
%defined in \S \ref{sec: case-by-case} should be substitute here. 
%We encourage the readers to skim some sample subsections of \S \ref{sec: case-by-case} 
%while reading this \S \ref{ssec: recipe case-by-case}. 

Looking over the whole construction, 
let us say in advance that there will be two types of technical complication: 
\begin{itemize}
\item Taking an etale cover of a parameter space, by which the sextics are endowed with certain markings. 
\item Taking a partial compactification of a parameter space, 
whose boundary is essentially the parameter space constructed in the previous step. 
\end{itemize}
The first type of construction is necessary for defining a family of higher Chow cycles globally. 
%Such a procedure is rather usual when constructing a family of higher Chow cycles. 
The second type of construction is a device for running our induction.

\subsubsection{Construction of $U_r\subset U_r'$}\label{sssec: U}

The first step of the construction is to take a parameter space $U_{r}\subset {\sextic}$ of a certain type of singular plane sextics. 
The sextics will have at most nodes and ordinary triple points as the singular points; 
when $r\leq 16$, they even have at most nodes. 
The locus $U_{r}$ is irreducible and locally closed in ${\sextic}$, and the family of sextics over $U_{r}$ is equisingular, 
i.e., they have the same type and number of singularities. 
The 2-elementary $K3$ surfaces associated to the sextics in $U_{r}$ have the desired invariant $(r, a, \delta)$, 
and the map $U_{r}\to {\moduli}$ to the moduli space will be dominant. 
%We write $U=U_{r}$ when we want to specify the invariant $r$. 

At the same time, except for the starting case $r=18$, we will also take a partial compactification 
$U_{r}\subset U_{r}'$ of $U_{r}$ inside its closure in ${\sextic}$ and observe that 
its boundary $U_{r}'-U_{r}$ is the parameter space $U_{r+1}$ 
we have already constructed in the step $r+1$. 
Thus we have $U_{r}'=U_{r}\cup U_{r+1}$. 

In the rest of \S \ref{sec: recipe} and \S \ref{sec: case-by-case}, 
for $U_r \subset U_r'$ and their coverings, 
we will use the same notations even after removing some exceptional locus (not the whole boundary).

\subsubsection{Construction of $\tilde{U}_{r}\subset \tilde{U}_{r}'$}\label{sssec: Utilde}

The second step of the construction is to take an etale cover $\tilde{U}_{r}\to U_{r}$ of $U_{r}$ 
by attaching a certain type of markings to the sextics. 
Thus $\tilde{U}_r$ parametrizes marked sextics, 
with $\tilde{U}_{r}\to U_{r}$ given by forgetting the markings. 
The type of marking depends on $r$, but in any case 
it induces a weak marking in the sense of Definition \ref{def: weak marking} in a specific manner. 
The space $\tilde{U}_{r}$ will be irreducible. 

At the same time, except for the starting case $r=18$, 
we will also take a partial compactification $\tilde{U}_{r}\subset \tilde{U}_{r}'$ of $\tilde{U}_{r}$ 
which still parametrizes marked sextics (degenerated further at the boundary) 
such that $\tilde{U}_{r}\to U_{r}$ extends to $\tilde{U}_{r}'\to U_{r}'$. 
We will have an etale map $\tilde{U}_{r+1} \to \tilde{U}_{r}'-\tilde{U}_{r}$ (often of degree $1$) 
which is given by converting the markings for $\tilde{U}_{r+1}$ to 
the limit markings for $\tilde{U}_{r}'-\tilde{U}_{r}$. 
The induced weak markings will agree, namely 
the family of weakly-marked sextics over $\tilde{U}_{r+1}$ coincides with 
the pullback of the one over $\tilde{U}_{r}'-\tilde{U}_{r}$ by this etale map. 

In this way, at the end of each subsection of \S \ref{sec: case-by-case}, 
we will have  parameter spaces which fit into the commutative diagram  
\begin{equation}\label{eqn: CD Ur inductive}
\xymatrix{
\tilde{U}_{r}  \ar@{^{(}-_>}[r] \ar[d] & \tilde{U}_{r}' \ar[d] & \tilde{U}_{r+1} \ar[d] \ar[l]  \\ 
U_{r}   \ar@{^{(}-_>}[r] & U_{r}'  & U_{r+1}  \ar@{_{(}-_>}[l] 
}
\end{equation}
%where $\tilde{U}_{r}'$ parametrizes marked sextics, 
%equisingular over $\hat{U}_{r}$, 
%and $\tilde{U}_{r+1}$ is an etale cover of $\tilde{U}_{r}'-\tilde{U}_{r}$. 

%The rest of the construction can be done uniformly. 
%We explain this process in full detail in the next \S \ref{ssec: recipe uniform}. 
%Thus the construction of the families $(X, \iota, \xi)\to \hat{U}_{r}$ will be completed at the end of \S \ref{ssec: recipe uniform}. 

\subsection{Uniform part of the construction}\label{ssec: recipe uniform}

Having constructed a parameter space $\tilde{U}=\tilde{U}_{r}$ of marked sextics, 
the next thing to do is to produce a family of $K3$ surfaces and higher Chow cycles on them. 
%This is the purpose of this \S \ref{ssec: recipe uniform}. 
The $K3$ family can be constructed over $\tilde{U}$ (in fact over $U$), 
but to construct a cycle family we need to take the base change to a double cover of $\tilde{U}$. 
Thus this \S \ref{ssec: recipe uniform} consists of three steps: 
\begin{enumerate}
\item Construction of a $K3$ family (\S \ref{sssec: K3 family}) 
\item Base change to a double cover $\hat{U}\to \tilde{U}$ (\S \ref{sssec: Uhat}) 
\item Construction of a higher Chow cycle family (\S \ref{sssec: cycle family}) 
\end{enumerate}

\subsubsection{Construction of a $K3$ family}\label{sssec: K3 family}

We have the universal plane sextic over ${\sextic}$ as a divisor of ${\sextic}\times {\plane}$. 
This divisor is linearly equivalent to $\mathcal{O}_{{\sextic}}(1) \boxtimes \mathcal{O}_{{\plane}}(6)$. 
Restricting this universal family over $U$ and pulling it back by $\tilde{U}\to U$, 
we obtain the universal family $B\to \tilde{U}$ of plane sextics over $\tilde{U}$ 
as a divisor of $\tilde{U}\times {\plane}$. 
Removing a general hyperplane section from $U$ if necessary (we denote the complement again by $U$), 
we can kill the contribution from $\mathcal{O}_{{\sextic}}(1)$ 
so that we may assume that 
$B$ is linearly equivalent to $p_{2}^{\ast}\mathcal{O}_{{\plane}}(6)$ 
where $p_2\colon \tilde{U}\times {\plane} \to {\plane}$ is the projection. 
Then we can take the double cover 
\begin{equation*}
\bar{\pi} : \bar{X}\to \tilde{U}\times {\plane} 
\end{equation*}
branched over $B$ 
inside the total space of $p_{2}^{\ast}\mathcal{O}_{{\plane}}(3)$. 
The projection $\bar{X}\to \tilde{U}$ is a family of singular $K3$ surfaces with 
$A_{1}$-points and $D_{4}$-points over the nodes and triple points of the sextic family $B\to \tilde{U}$ respectively. 
Since $B\to \tilde{U}$ is equisingular, $\bar{X}\to \tilde{U}$ is also equisingular. 
Hence we can take the minimal resolution in family: 
\begin{equation*}
X\to \bar{X}. 
\end{equation*}
The covering transformation of $\bar{X}\to \tilde{U}\times {\plane}$ 
extends to an involution $\iota$ of $X$ by the minimality of the $K3$ fibers. 
On each fiber, it is non-symplectic. 
In this way we obtain a family 
\begin{equation*}
(X, \iota) \to \tilde{U} 
\end{equation*}
of 2-elementary $K3$ surfaces over $\tilde{U}$.

%We observe that the construction of the family of \textit{singular} double covers 
%can be extended over the partial compactification $\tilde{U}'=\tilde{U}'_{r}$ 
%because the universal sextic is still defined over $U'$. 
%Its restriction to the boundary $\tilde{U}_{r}' - \tilde{U}_{r}$ is 
%the family of singular $K3$ surfaces we have constructed over $\tilde{U}_{r+1}$ 
%(or its etale descend when $\tilde{U}_{r+1}\to \tilde{U}_{r}' - \tilde{U}_{r}$ has degree $>1$). 

\subsubsection{The double cover $\hat{U}\to \tilde{U}$}\label{sssec: Uhat}

The space $\tilde{U}$ parametrizes plane sextics with weak marking. 
In order to define a family of higher Chow cycles, 
we want to promote the weak markings to strong markings globally. 
This is made possible after taking a double cover of $\tilde{U}$ as follows. 

By our weak marking over $\tilde{U}$, 
we have two specified nodes of $B_{u}$ for each $u\in \tilde{U}$. 
They form two disjoint sections of $B\to \tilde{U}$, 
say $N_1, N_{2}\subset B$. 
The lines joining these two nodes form a smooth ${\proj}^{1}$-bundle over $\tilde{U}$, 
say $L\subset \tilde{U}\times {\plane}$. 
After shrinking $U$ to a Zariski open set, we may assume that 
these nodes satisfy the genericity Condition \ref{condition: genericity}. 
Then $\bar{\pi}^{-1}(N_1)$ and $\bar{\pi}^{-1}(N_2)$ are families of $A_{1}$-singularities of $\bar{X}\to \tilde{U}$, 
and $\bar{\pi}^{-1}(L)$ is a family of two-nodal rational curves 
with nodes at $\bar{\pi}^{-1}(N_1)$ and  $\bar{\pi}^{-1}(N_2)$. 
In particular, for each $u\in \tilde{U}$, the curve $\bar{\pi}^{-1}(L)_{u}$ has two branches at $\bar{\pi}^{-1}(N_1)_{u}$. 
Choosing one of them is the same as promoting our weak marking to a strong marking. 
There are two such choices. 
This defines an etale double cover 
\begin{equation*}
\hat{U} \to \tilde{U} 
\end{equation*}
of $\tilde{U}$. 
By construction, the sextic family pulled back to $\hat{U}$ is globally endowed with a strong marking. 
We can show that $\hat{U}$ is connected (Remark \ref{rmk: connected}), 
but logically this property is not necessary for our purpose.  

We pullback $X\to \bar{X}\to \tilde{U}\times {\plane}$ by $\hat{U}\to \tilde{U}$, 
and abusing notation, denote it again by 
\begin{equation*}
\pi : X \to \bar{X} \stackrel{\bar{\pi}}{\to} \hat{U}\times {\plane}. 
\end{equation*}
We also use the same notations 
\begin{equation*}
B, \; L, \; N_{i} \: \:  \subset \: \hat{U}\times {\plane} 
\end{equation*}
for the pullback of $B, L, N_{i} \subset \tilde{U}\times {\plane}$ by $\hat{U}\to \tilde{U}$. 
%In what follows, these notations are used almost only for the families over $\tilde{U}$.  

Before going ahead, let us observe that 
$\hat{U}$ can be extended to a double cover $\hat{U}'\to \tilde{U}'$ of 
the partial compactification of $\tilde{U}$ defined in \S \ref{sssec: Utilde}. 
Indeed, construction of the singular double cover $\bar{X}$  
can be extended over $\tilde{U}'$, 
as the universal sextic is still defined over $\tilde{U}'$ and hence over $\tilde{U}'$. 
This produces a family of singular $K3$ surfaces over $\tilde{U}'$ 
which acquire additional singularities at the boundary $\tilde{U}'-\tilde{U}$. 
%The further degenerated family over the boundary is an etale descend of the family over $\tilde{U}_{r+1}$. 
Then $N_{1}, N_{2}$ and $L$ extend over $\tilde{U}'$ by construction, 
and they are disjoint from the new singularities. 
Thus the family $\bar{\pi}^{-1}(L)$ of two-nodal rational curves extends over $\tilde{U}'$, 
so we can define an etale double cover 
\begin{equation*}
\hat{U}'\to \tilde{U}' 
\end{equation*} 
in the same way as for $\hat{U}\to \tilde{U}$. 
%By construction, $\hat{U}'$ parametrizes strongly-marked sextics. 

At the boundary $\tilde{U}' - \tilde{U}=\tilde{U}_{r}'-\tilde{U}_{r}$,  
this construction of double cover takes the same process as $\hat{U}_{r+1}\to \tilde{U}_{r+1}$ 
by the coincidence of weak markings. 
This means that the map $\hat{U}_{r+1}\to \tilde{U}_{r+1}$ is the base change of  
$\hat{U}_{r}' - \hat{U}_{r} \to \tilde{U}_{r}' - \tilde{U}_{r}$ 
by the etale map $\tilde{U}_{r+1}\to \tilde{U}_{r}' - \tilde{U}_{r}$. 
In this way we obtain the extended parameter space $\hat{U}_{r}'$ of strongly-marked sextics 
with inductive structure   
\begin{equation}\label{eqn: extend Utilde}
\hat{U}_{r}  \hookrightarrow \hat{U}_{r}'  \leftarrow \hat{U}_{r+1}, 
\end{equation}
where $\hat{U}_{r+1}$ is an etale cover of $\hat{U}_{r}'-\hat{U}_{r}$ (often of degree $1$). 
The pullback of the family of strongly-marked sextics over $\tilde{U}_{r}' - \tilde{U}_{r}$ 
by $\tilde{U}_{r+1}\to \tilde{U}_{r}' - \tilde{U}_{r}$ is the one already constructed over $\tilde{U}_{r+1}$.

\begin{remark}\label{rmk: connected}
The connectedness of $\hat{U}$ can be seen inductively as follows. 
The connectedness in the starting case $r=18$ follows from \cite{Sa} 
(see Remark \ref{rmk: irreducible 18}). 
Then, by \eqref{eqn: extend Utilde}, the connectedness of $\hat{U}_{r+1}$ implies that of 
the boundary $\hat{U}_{r}'-\hat{U}_{r}$ of $\hat{U}_{r}'$. 
Since $\tilde{U}_{r}'$ is connected, this implies that $\hat{U}_{r}'$ is connected. 
\end{remark}

\subsubsection{Construction of a cycle family}\label{sssec: cycle family} 

After the base change to $\hat{U}$, 
we can now define a family of higher Chow cycles globally. 
Let $Z_{0}\subset X$ be the strict transform of $\bar{\pi}^{-1}(L)\subset \bar{X}$ in $X$ 
and let $Z_{1}=\pi^{-1}(N_{1})$. 
Then $Z_{0}, Z_{1}$ are smooth families of $(-2)$-curves meeting transversely at two points in each fiber. 
These two points correspond to the two branches of $\bar{\pi}^{-1}(L)_{u}$ at $\bar{\pi}^{-1}(N_{1})_{u}$. 
By our construction of $\hat{U}$, these two branches are globally distinguished over $\hat{U}$, 
so we have 
\begin{equation*}
Z_{0} \cap Z_{1} = S\!_{0} \sqcup S\!_{\infty} 
\end{equation*}
with each $S\!_{\ast}$ being a section of both $Z_{i}\to \hat{U}$. 
This shows in particular that $Z_{0}, Z_{1}$ are smooth ${\proj}^1$-bundles over $\hat{U}$ in the Zariski topology. 
After shrinking $\tilde{U}$ to a Zariski open set, 
we can take isomorphisms 
\begin{equation*}
i_{0} : Z_{0} \to {\proj}^1\times \hat{U}, \qquad
i_{1} : Z_{1} \to {\proj}^1\times \hat{U}, 
\end{equation*}
over $\hat{U}$ such that 
\begin{equation*}
i_0(S\!_0)=(0)\times \hat{U}, \; \;   
i_0(S\!_{\infty})=(\infty)\times \hat{U},  \; \;    
i_1(S\!_0)=(\infty)\times \hat{U},  \; \;    
i_1(S\!_{\infty})=(0)\times \hat{U}. 
\end{equation*} 
Let $z$ be the standard rational function on ${\proj}^1$ with ${\rm div}(z)=(0)-(\infty)$. 
Then 
\begin{equation*}
f_{0} = i_{0}^{\ast}z, \qquad f_{1} = i_{1}^{\ast}z  
\end{equation*}
are rational functions on $Z_{0}$, $Z_{1}$ with 
${\rm div}(f_{0})=S\!_{0}-S\!_{\infty}$ and ${\rm div}(f_{1})=S\!_{\infty}-S\!_{0}$ respectively. 
In this way we obtain a family 
\begin{equation*}
\xi = (Z_0, f_0) + (Z_1, f_1) 
\end{equation*}
of higher Chow cycles on the family $X\to \hat{U}$ of smooth $K3$ surfaces. 
The rest of \S \ref{sec: recipe} is devoted to 
proving nontriviality of these cycles (Theorem \ref{thm: nontrivial}).

\subsection{A degeneration lemma}\label{ssec: degeneration lemma}

In order to run our induction, 
we need a lemma with which we are able to deduce nontriviality over $\hat{U}$ 
from that over the boundary $\hat{U}'-\hat{U}$. 
Our purpose in this \S \ref{ssec: degeneration lemma} is to prepare such a lemma. 
Let us state this lemma in a self-contained way. 

Let $\Delta$ be the unit disc and $\Delta^{\ast}=\Delta\backslash \{ 0 \}$. 
Let $B\subset \Delta\times{\plane}$ be a family of plane sextics with at most simple singularities over $\Delta$ 
which is equisingular over $\Delta^{\ast}$. 
We assume that $B$ has two specified families $N_1, N_2\subset B$ of nodes which are disjoint. 
Thus $N_1, N_2$ are sections of $B\to \Delta$ such that 
$(N_1)_t$, $(N_2)_t$ are distinct nodes of $B_{t}$ for every $t\in \Delta$ (including $t=0$). 
Let $L_t\subset \{ t \}\times {\plane}$ be the line joining $(N_1)_t$ and $(N_2)_t$, 
and let $L=\cup_{t\in \Delta}L_{t}$. 
We assume that the genericity Condition \ref{condition: genericity} is satisfied for every $t\in \Delta$. 
In particular, the new singularity over $t=0$ appears outside $L_{0}$. 
Let $\bar{\pi}\colon \bar{X}\to \Delta\times{\plane}$ be the double cover branched over $B$. 
Then $\bar{\pi}^{-1}(L)$ is a family of two-nodal rational curves.  
Since $\Delta$ is simply connected, the two branches of $\bar{\pi}^{-1}(L_{t})$ 
at the node $\bar{\pi}^{-1}((N_{1})_{t})$ are distinguished over $\Delta$. 
Choosing one of them, we obtain a strong marking of the sextic family $B\to \Delta$, 
say $\mu=(\mu_{t})_{t}$. 

For each $t \in \Delta$, let $(X_{t}, \iota_{t})$ be the 2-elementary $K3$ surface 
obtained as the minimal resolution of $\bar{X}_{t}$. 
We consider the anti-invariant regulator  
\begin{equation*}\label{eqn: def nu-}
\nu_{-}(B_t, \mu_t) \: \: \in J(X_t, \iota_t) 
\end{equation*}
of the higher Chow cycles attached to $(B_t, \mu_t)$ as defined in \eqref{eqn: anti-invariant regulator}. 
When $B_t$ degenerates further at $t=0$, 
the Jacobian $J(X_0, \iota_0)$ has in general smaller dimension than $J(X_t, \iota_t)$ for $t\ne 0$. 
%(Recall from \S \ref{ssec: from sextic to cycle} that $\nu_{-}(B_t, \mu_t)$ depends only on $(B_t, \mu_t)$.) 
%We should note here that although the singular $K3$ surfaces $\bar{X}_{t}$ form the family $\bar{X}\to \Delta$ over $\Delta$, 
%their minimal resolutions $X_t$ do not necessarily do so. 

We can now state our degeneration lemma. 

\begin{lemma}\label{lem: degeneration}
If $\nu_{-}(B_0, \mu_0)$ is non-torsion, then $\nu_{-}(B_t, \mu_t)$ is non-torsion for very general $t\in \Delta^{\ast}$, 
i.e., for $t\in \Delta^{\ast}$ in the complement of countably many points. 
\end{lemma}

\begin{proof}
By a theorem of Tjurina \cite{Tj}, after taking a cyclic base change of $\bar{X}\to \Delta$, 
we can take its simultaneous minimal resolution. 
Thus we have a commutative diagram 
\begin{equation*}
\xymatrix{
X \ar[rd]_-{f} \ar[r]^-{p_2} & \bar{X}' \ar[d] \ar[r]^-{p_1} & \bar{X} \ar[r]^-{\bar{\pi}} \ar[d] & \Delta\times {\plane} \ar[ld] \\ 
  & \Delta \ar[r] & \Delta  &  
}
\end{equation*} 
where $\Delta \to \Delta$ is given by $t\mapsto t^{N}$, 
the middle square is cartesian, 
and $p_{2}$ is a simultaneous minimal resolution of the family of ADE singularities. 
The end product $f\colon X\to \Delta$ is a smooth family of $K3$ surfaces. 
(Note, however, that the involution on $X|_{\Delta^{\ast}}$ does not extend over $X$ in general.) 

Let $\pi= \bar{\pi}\circ p_1 \circ p_2$. 
Then $Z_1=\pi^{-1}(N_1)$ is the family of $(-2)$-curves over 
the family $(\bar{\pi}\circ p_{1})^{-1}(N_1)$ of $A_1$-points. 
We also let $Z_{0}\subset X$ be the strict transform of  $(\bar{\pi}\circ p_{1})^{-1}(L)$ in $X$. 
Then $Z_1$ is a family of $(-2)$-curves which intersects with $Z_{0}$ at two points in each fiber. 
Thus $Z_0\cap Z_1=S\!_0\sqcup S\!_{\infty}$ where 
each $S\!_{\ast}$ is a section of $X\to \Delta$. 
Here the indices $0$, $\infty$ are assigned according to our strong marking. 
As in \S \ref{sssec: cycle family}, we take meromorphic functions $f_0, f_1$ on $Z_0, Z_1$ such that 
${\rm div}(f_0)=S\!_0-S\!_{\infty}$ and ${\rm div}(f_1)=S\!_{\infty}-S\!_0$ respectively. 
Thus we obtain the analytic family 
\begin{equation*}\label{eqn: cycle family local analytic}
\xi = (Z_0, f_0) + (Z_1, f_1) 
\end{equation*}
of higher Chow cycles on the smooth $K3$ family $f\colon X\to \Delta$. 
%(whereas the involution degenerates at $t=0$). 
By construction, the fiber $\xi_{t}$ over $t\in \Delta$ 
is a higher Chow cycle attached to the strongly-marked sextic $(B_{t^{N}}, \mu_{t^{N}})$. 
It is sufficient to show that $\nu_{-}(\xi_{t})$ is non-torsion for very general $t$. 

Next we extend the family of Jacobians over $\Delta^{\ast}$ 
to a smooth family of generalized complex tori over $\Delta$. 
Since $B|_{\Delta^{\ast}}$ is equisingular, 
the deformation type of $(X_{t}, \iota_{t})$ is stable for $t\ne 0$. 
This implies that the primitive sublattices $H_{-}(X_t, \iota_t)$ of $H^2(X_t, {\Z})$ for $t\ne 0$ 
form a sub local system of $R^2f_{\ast}{\Z}|_{\Delta^{\ast}}$. 
Since $R^2f_{\ast}{\Z}$ is a trivial local system over $\Delta$, 
we find that 
$\cup_{t\ne 0}H_{-}(X_t, \iota_t)$ 
extends to a sub local system of $R^2f_{\ast}{\Z}$ over $\Delta$. 
We denote it by $\mathcal{H}_{-}\subset R^2f_{\ast}{\Z}$. 
Let $(\mathcal{H}_{-})_{0}$ be the stalk of $\mathcal{H}_{-}$ at $t=0$. 

\begin{claim}\label{claim: L-}
We have $H_{-}(X_0, \iota_0)\subset (\mathcal{H}_{-})_{0}$ in $H^2(X_0, {\Z})$. 
\end{claim}

\begin{proof}
Since both $H_{-}(X_0, \iota_0)$ and $(\mathcal{H}_{-})_{0}$ are primitive sublattices of $H^2(X_0, {\Z})$, 
it suffices to verify this inclusion after taking the tensor product with ${\Q}$. 
Let $t\ne 0$. 
By Lemma \ref{lem: H- vs IH}, $H_-(X_{t}, \iota_{t})_{{\Q}}$ is naturally isomorphic to 
the primitive part of $IH^2(\bar{X}'_{t}, {\Q})=H^2(\bar{X}'_{t}, {\Q})$. 
Therefore the natural isomorphism 
$H^2(X_0, {\Q})\to H^2(X_t, {\Q})$ 
induces the commutative diagram 
\begin{equation*}
\xymatrix{
H_{-}(X_{0}, \iota_{0})_{{\Q}}  \ar[r]^-{\simeq} \ar@{.>}[d]  & H^2(\bar{X}_{0}', {\Q})_{prim}  \ar@{^{(}-_>}[r]  \ar@{^{(}-_>}[d] 
& H^2(\bar{X}_{0}', {\Q})  \ar@{^{(}-_>}[d]  \ar@{^{(}-_>}[r]  & H^2(X_{0}, {\Q}) \ar[d]^-{\simeq} \\ 
H_{-}(X_{t}, \iota_{t})_{{\Q}}  \ar[r]^-{\simeq} & H^2(\bar{X}_{t}', {\Q})_{prim}  \ar@{^{(}-_>}[r] 
& H^2(\bar{X}_{t}', {\Q})  \ar@{^{(}-_>}[r] & H^2(X_{t}, {\Q}) 
}
\end{equation*}
where the two middle vertical maps are the specialization maps 
for the intersection cohomology. 
Since $(\mathcal{H}_{-})_{0}\otimes {\Q}$ is the image of $H_{-}(X_t, \iota_{t})_{{\Q}}$ in $H^2(X_0, {\Q})$, 
we obtain $H_{-}(X_0, \iota_0)_{{\Q}}\subset (\mathcal{H}_{-})_{0}\otimes {\Q}$. 
\end{proof}

We go back to the proof of Lemma \ref{lem: degeneration}. 
Let 
\begin{equation*}
\mathcal{J} = \mathcal{J}(R^2f_{\ast}{\Z}), \quad 
\mathcal{J}_{-} = \mathcal{J}(\mathcal{H}_{-}^{\vee}), 
\end{equation*}
be the Jacobian fibrations over $\Delta$ attached to 
the variations of Hodge structures $R^2f_{\ast}{\Z}$ and $\mathcal{H}_{-}^{\vee}$ respectively 
(cf.~\S \ref{ssec: normal function} and \S \ref{ssec: Jacobian}). 
The orthogonal projection 
$R^2f_{\ast}{\Z} = (R^2f_{\ast}{\Z})^{\vee} \twoheadrightarrow \mathcal{H}_{-}^{\vee}$ 
defines a surjective homomorphism 
$\mathcal{J}\twoheadrightarrow \mathcal{J}_{-}$ 
over $\Delta$. 
When $t\ne0$, we have 
$(\mathcal{J}_{-})_{t}=J(X_t, \iota_{t})$. 
%by the definition \eqref{eqn: J(X,i)} of $J(X, \iota)$. 
When $t=0$, the projection $J(X_0)\to J(X_0, \iota_0)$ factors through $(\mathcal{J}_{-})_{0}$ by Claim \ref{claim: L-}: 
\begin{equation}\label{eqn: J-0}
J(X_0) = \mathcal{J}_{0} \twoheadrightarrow (\mathcal{J}_{-})_{0} \twoheadrightarrow J(X_0, \iota_0). 
\end{equation} 

The normal function $\nu(\xi)$ of our cycle family $\xi$ is 
a holomorphic section of $\mathcal{J}\to \Delta$ over $\Delta$. 
Let 
\begin{equation*}
\nu_{-}(\xi) :  \Delta\to \mathcal{J}_{-} 
\end{equation*} 
be the composition of $\nu(\xi)$ with the projection $\mathcal{J}\to \mathcal{J}_{-}$. 
This is a holomorphic section of $\mathcal{J}_{-} \to \Delta$ over $\Delta$. 
Then we have 
\begin{equation}\label{eqn: nu-}
\nu_{-}(\xi)(t) = \nu_{-}(\xi_{t})  \quad \in \;  (\mathcal{J}_{-})_{t}=J(X_t, \iota_{t}) 
\end{equation}
for $t\ne 0$ by the definition \eqref{eqn: anti-invariant regulator} of the anti-invariant regulator.  
On the other hand, when $t=0$, the image of $\nu_{-}(\xi)(0)$ by the projection 
$(\mathcal{J}_{-})_{0}\to J(X_0, \iota_0)$ 
is equal to $\nu_{-}(\xi_0)$. 
Indeed, by \eqref{eqn: J-0}, 
the image of $\nu_{-}(\xi)(0)$ in $J(X_0, \iota_0)$ is the image of 
$\nu(\xi)(0)=\nu(\xi_{0})$ in $J(X_0, \iota_0)$, 
and hence is $\nu_{-}(\xi_{0})$. 

We can now complete the proof of Lemma \ref{lem: degeneration}. 
Our assumption is that $\nu_{-}(\xi_0)\in J(X_0, \iota_0)$ is non-torsion. 
By the above consideration, 
this implies that $\nu_{-}(\xi)(0)\in (\mathcal{J}_{-})_0$ is non-torsion. 
Since $\nu_{-}(\xi)$ is a holomorphic section of the smooth fibration $\mathcal{J}_{-}\to \Delta$ 
and since the torsion points of a generalized complex torus are countable, 
we see that the locus of points $t\in \Delta$ where $\nu_{-}(\xi)(t)$ is torsion is either 
the whole $\Delta$ or a union of countably many proper analytic subsets, namely countably many points. 
Since $\nu_{-}(\xi)(0)$ is non-torsion, the former case does not occur. 
By \eqref{eqn: nu-}, we conclude that $\nu_{-}(\xi_t)$ is non-torsion for very general $t$.  
\end{proof}

\begin{remark}
When $X_0$ has minimal Picard rank, i.e., ${\rm NS}(X_0)=H_{+}(X_0, \iota_{0})$, 
Claim \ref{claim: L-} can be seen more easily. 
Indeed, writing $H_+=H_+(X_{t_0}, \iota_{t_{0}})$ for some reference $t_{0}\ne 0$, 
the family $X|_{\Delta^{\ast}}$ is $H_+$-polarized. 
Letting $t\to 0$, we still have $(H_+, H^{2,0}(X_0))\equiv 0$ at $t=0$. 
Hence $H_+\subset {\rm NS}(X_0)= H_{+}(X_0, \iota_{0})$. 
Taking the orthogonal complement, we see that 
$H_{-}(X_0, \iota_{0})\subset H_{-}$. 
In fact, for running our induction, Lemma \ref{lem: degeneration} for $1$-parameter degenerations 
with such a central fiber $X_0$ is sufficient. 
\end{remark}

\subsection{Nontriviality}\label{ssec: nontriviality}

Let $(X, \iota, \xi)\to \hat{U}$ 
be the family of 2-elementary $K3$ surfaces and higher Chow cycles on them constructed 
in \S \ref{ssec: recipe uniform}. 
We can now prove our main result. 

\begin{theorem}[Theorem \ref{thm: intro}]\label{thm: nontrivial}
The indecomposable part $(\xi_{u})_{ind}$ of $\xi_{u}$ is non-torsion for very general $u\in \hat{U}$. 
\end{theorem}

We first reduce Theorem \ref{thm: nontrivial} to the non-torsionness of the anti-invariant normal function. 
Let $\mathcal{J}_{-}\to \hat{U}$ be the fibration of Jacobians of $(X, \iota)\to \hat{U}$, 
and let 
\begin{equation*}
\nu_{-}(\xi) : \hat{U}\to \mathcal{J}_{-}, \quad 
\nu_{-}(\xi)(u) = \nu_{-}(\xi_{u}),  
\end{equation*}
be the anti-invariant normal function of $\xi$. 
%What we actually prove is the non-torsionness of $\nu_{-}(\xi)$. 

\begin{theorem}\label{thm: nontrivial normal function}
$\nu_{-}(\xi)(u)\in (\mathcal{J}_{-})_{u}$ is non-torsion for very general $u\in \hat{U}$. 
\end{theorem}

Theorem \ref{thm: nontrivial} can be derived from Theorem \ref{thm: nontrivial normal function} as follows. 

\begin{proof}[(Proof of Theorem \ref{thm: nontrivial})]
Let $\hat{U}^{\circ}\subset \hat{U}$ be the locus where $\nu_{-}(\xi)(u)$ is non-torsion. 
Theorem \ref{thm: nontrivial normal function} assures that $\hat{U}^{\circ}$ is non-empty; 
it contains the complement of countably many proper analytic subsets. 
Let $\hat{U}^{+}\subset \hat{U}$ 
be the locus where ${\rm NS}(X_{u})=H_{+}(X_u, \iota_u)$. 
Since $\hat{U}$ dominates the moduli space of 2-elementary $K3$ surfaces of this type, 
$\hat{U}^{+}$ is the complement of countably many divisors (see \S \ref{ssec: K3 classify}). 
If $u\in \hat{U}^{+}$, then 
$J(X_u, \iota_u)=J_{tr}(X_u)$ as explained in \S \ref{ssec: Jacobian}, 
and hence $\nu_{-}(\xi)(u)=\nu_{tr}((\xi_{u})_{ind})$ for such $u$. 
Therefore, if $u\in \hat{U}^{\circ}\cap \hat{U}^{+}$, then $\nu_{tr}((\xi_{u})_{ind})$ is non-torsion. 
It follows that $(\xi_u)_{ind}$ is non-torsion for such $u$. 
%In this way Theorem \ref{thm: nontrivial} is deduced from Theorem \ref{thm: nontrivial normal function}. 
\end{proof}

%\begin{remark}
%The nontriviality of $\nu_{-}(\xi)$ implies that $\nu(\xi)\not\equiv 0$ over $\hat{U}$. 
%Since $\nu(\nu)$ is holomorphic, we see that 
%$\nu(\xi_{u})\ne 0$ (and so $\xi_{u}\ne 0$ in ${\rm CH}^2(X_{u}, 1)$)  
%for $u\in \hat{U}$ in the complement of a proper analytic subset. 
%\end{remark}

Now we prove Theorem \ref{thm: nontrivial normal function}. 

\begin{proof}[(Proof of Theorem \ref{thm: nontrivial normal function})]
We proceed by induction on the dimension $20-r$ of the moduli space. 
We write $\hat{U}=\hat{U}_{r}$ for indicating the invariant $r$. 

The assertion in the starting case $r=18$ was essentially proved in \cite{Sa}, 
where the family was constructed from certain bidegree $(4, 4)$ curves on ${\proj}^1\times{\proj}^1$. 
Translation of the result of \cite{Sa} to the present situation is explained in \S \ref{sssec: translation 18}. 
%(see especially around \eqref{eqn: nontrivial 18}). 

Suppose that we have proved the assertion in the step $r+1$, i.e., for $\hat{U}_{r+1}$. 
Let $\hat{U}_{r}\subset \hat{U}_{r}'$ be the extended parameter space constructed 
in the second half of \S \ref{sssec: Uhat}. 
We take a small arc $\lambda \colon \Delta\to \hat{U}_{r}'$ 
such that $\lambda(\Delta^{\ast})\subset \hat{U}_{r}$ and 
$\lambda(0)\in \hat{U}_{r}'-\hat{U}_{r}$ is the image of a very general point of $\hat{U}_{r+1}$ 
(as prescribed in the step $r+1$) 
by the etale map $\hat{U}_{r+1} \to \hat{U}_{r}'-\hat{U}_{r}$. 
Since the anti-invariant regulator depends only on the strongly-marked sextic (Lemma \ref{lem: anti-inv regulator}), 
the assumption of induction implies that  
the hypothesis of Lemma \ref{lem: degeneration} is fulfilled 
for the induced family of strongly-marked sextics over $\Delta$. 
Therefore we can apply Lemma \ref{lem: degeneration} to see that 
$\nu_{-}(\xi)(\lambda(t))$ is non-torsion for very general $t\in \Delta^{\ast}$. 

Now, since $\nu_{-}(\xi)$ is a holomorphic section of the Jacobian fibration $\mathcal{J}_{-}$, 
the locus of points $u\in \hat{U}_{r}$ where $\nu_{-}(\xi)(u)$ is torsion is either 
the whole $\hat{U}_{r}$ or a union of countably many proper analytic subsets. 
By the above consideration, the first case does not occur. 
This implies our assertion in the step $r$. 
\end{proof}

\begin{remark}\label{rmk: rank 2}
If we use the $(-2)$-curve $Z_2$ over $q_2$ in place of $Z_1$, 
we obtain a higher Chow cycle with support $Z_0+Z_2$ (see Figure 1). 
Thus, after a further base change, 
we obtain a second family of higher Chow cycles, say $\xi'$. 
Then the proof of Theorem \ref{thm: nontrivial normal function} can be adapted to 
show that $\nu_{-}(\xi)(u)$ and $\nu_{-}(\xi')(u)$ are linearly independent 
over ${\Q}$ for very general $u$: 
the starting case $r=18$ was essentially done in \cite{Sa},  
and the degeneration lemma can be easily generalized to the linear independence version. 
Therefore the indecomposable parts of $\xi_{u}$ and $\xi_{u}'$ 
are linearly independent over ${\Q}$ for very general $u$. 
In particular, ${\rm CH}^{2}(X_u, 1)_{ind}$ has rank $\geq 2$. 
This gives a strengthening of Theorem \ref{thm: nontrivial}. 
Since we want to keep our presentation simple, 
we leave the detail to the readers. 
\end{remark}

\section{Case-by-case constructions}\label{sec: case-by-case}

In this section, for each $3\leq r \leq 18$, 
we construct parameter spaces $\tilde{U}_{r}\to U_{r}$ 
and their partial compactifications $\tilde{U}_{r}' \to U_{r}'$ (when $r<18$) 
as announced in \S \ref{ssec: recipe case-by-case}. 
These spaces are to be substituted in \S \ref{sec: recipe} 
as the sources for producing families of higher Chow cycles. 

The space $U_r$ will be a locus in ${\sextic}$ parametrizing 
an equisingular family of plane sextics, 
and its covering $\tilde{U}_{r}$ endows the sextics with a certain type of markings. 
The type of marking depends on $r$, but it will be one of the following or sometimes their mixture: 
\begin{itemize}
\item labeling of some nodes of the sextic 
\item labeling of the irreducible components of the sextic 
\end{itemize}
In any case, a weak marking in the sense of Definition \ref{def: weak marking} 
will be induced in a specific way. 
In practice, we will often define $\tilde{U}_{r}$ first, 
and then define $U_{r}$ as the image of the natural map $\tilde{U}_{r}\to {\sextic}$.  

The partial compactification $U_{r}'$ of $U_{r}$ will be taken inside the closure of $U_{r}$ in ${\sextic}$, 
and its boundary $U_{r}'-U_r$ will coincide with $U_{r+1}$. 
The boundary of the partial compactification $\tilde{U}_{r}'$ of the covering $\tilde{U}_{r}$ 
parametrizes degenerated sextics endowed with limit of markings on $\tilde{U}_{r}$. 
We will have an etale map $\tilde{U}_{r+1} \to \tilde{U}_{r}'-\tilde{U}_{r}$ 
which converts the markings on $\tilde{U}_{r+1}$ to the limit markings on $\tilde{U}_{r}'-\tilde{U}_{r}$. 
The induced weak markings will agree. 
This shows that the relevant spaces fit into the commutative diagram \eqref{eqn: CD Ur inductive}. 
%\begin{equation*}
%\xymatrix{
%\tilde{U}_{r}  \ar@{^{(}-_>}[r] \ar[d] & \tilde{U}_{r}' \ar[d] & \tilde{U}_{r+1} \ar[d] \ar[l]  \\ 
%U_{r}   \ar@{^{(}-_>}[r] & U_{r}'  & U_{r+1}  \ar@{_{(}-_>}[l] 
%}
%\end{equation*}
In fact, for running our induction in \S \ref{ssec: nontriviality}, 
it is sufficient to construct partial compactifications only locally. 
Here we give a global construction in order to have a better understanding.

\subsection{The case $r=18$}\label{ssec: r=18}

In this subsection, 
we define the parameter spaces in the case $r=18$ 
and explain that the assertion of Theorem \ref{thm: nontrivial normal function} in this case 
follows from the result of \cite{Sa}. 

\subsubsection{The parameter spaces}\label{sssec: r=18}

We first define $\tilde{U}_{18}$ as the codimension $2$ locus in ${\lin}^{6}$ parametrizing 
six ordered distinct lines $(L_1, \cdots, L_6)$ such that 
$L_1, L_2, L_3$ intersect at one point, 
$L_4, L_5, L_6$ intersect at one point, 
and no other three of $L_1, \cdots, L_6$ intersect at one point. 
(See Figure 2.) 
Then we let $U_{18}\subset {\sextic}$ be the image of the natural map 
\begin{equation*}
\tilde{U}_{18} \to {\sextic}, \qquad (L_1, \cdots, L_6)\mapsto L_1+ \cdots +L_6. 
\end{equation*}
In \cite{Ma1} \S 9.5, it is proved that the 2-elementary $K3$ surfaces associated to 
the sextics $L_1+ \cdots + L_6$ in $U_{18}$ have invariant $(r, a, \delta)=(18, 4, 0)$, 
and $U_{18}$ dominates the moduli space $\mathcal{M}_{18,4,0}$. 

The projection $\tilde{U}_{18}\to U_{18}$ forgets labeling of the lines $L_1, \cdots, L_6$. 
This is a Galois cover with Galois group $\frak{S}_2\ltimes (\frak{S}_{3}\times \frak{S}_{3})$. 
For $(L_1, \cdots, L_6)\in \tilde{U}_{18}$ we define a weak marking of $L_1+ \cdots + L_6$ 
in the sense of Definition \ref{def: weak marking} by selecting the nodes as 
\begin{equation}\label{eqn: weak marking 18}
q_1 = L_1 \cap L_4, \quad q_2=L_2\cap L_5. 
\end{equation}
Imposing the genericity Condition \ref{condition: genericity}, 
we need to shrink $\tilde{U}_{18}$ (and $U_{18}$) to a Zariski open set. 
As explained at the end of \S \ref{sssec: U}, 
we will not change the notations even after removing some appropriate locus like this. 
This process will take place in every subsequent subsection, 
and we will not repeat this announcement explicitly. 

\begin{figure}[h]\label{figure: r18}
\includegraphics[height=45mm, width=65mm]{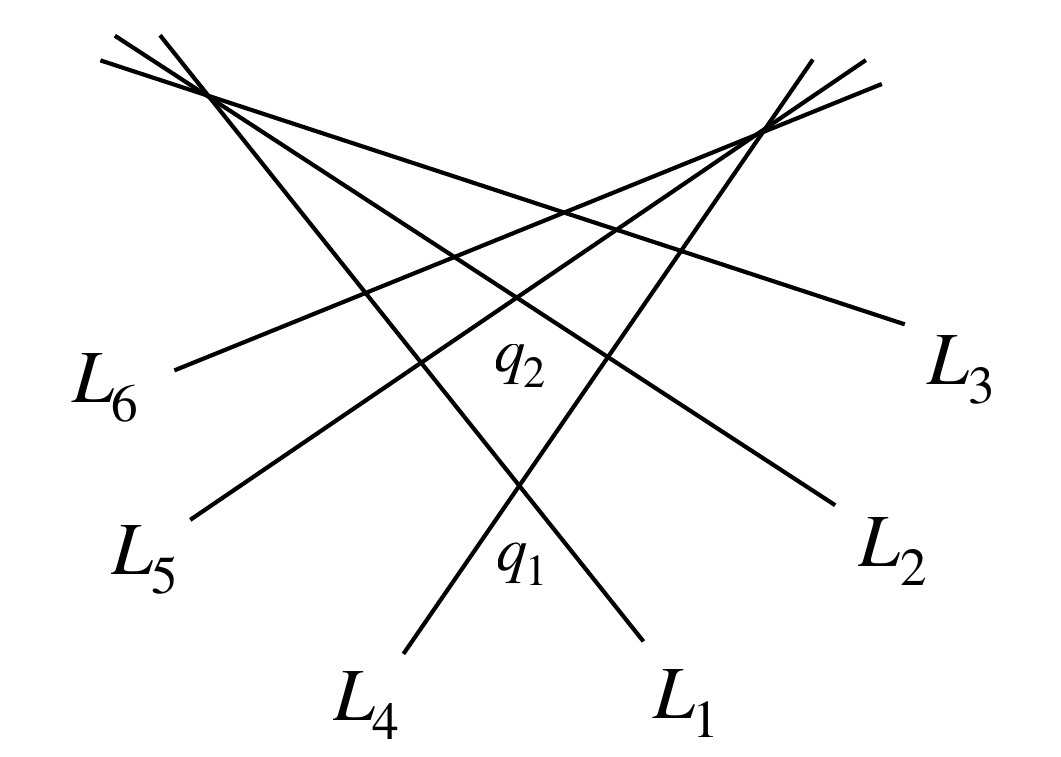}
\caption{$r=18$}
\end{figure}

\subsubsection{Translation to ${\proj}^1\times {\proj}^1$}\label{sssec: translation 18}

Let $\hat{U}_{18}\to \tilde{U}_{18}$ be the double cover 
constructed by the recipe of \S \ref{sssec: Uhat}, and 
$\xi$ be the family of higher Chow cycles over $\hat{U}_{18}$ constructed by the recipe of \S \ref{sssec: cycle family}. 
We shall explain that the result of \cite{Sa} implies 
the assertion of Theorem \ref{thm: nontrivial normal function} in this case. 

We take the geometric quotient of $\tilde{U}_{18}$ by ${\rm PGL}_3$. 
Its effect is to normalize four lines, say $L_1, L_2, L_4, L_5$, 
while $L_3$ and $L_6$ vary. 
Let $p_1=L_1\cap L_2$ and $p_2=L_4\cap L_5$. 
By blowing up $p_1$, $p_2$ and blowing down the strict transform of the line $\overline{p_{1}p_{2}}$, 
we pass from ${\plane}$ to ${\proj}^1\times {\proj}^1$. 
The lines $L_1, \cdots, L_6$ and the $(-1)$-curves $E_{1}$, $E_{2}$ over $p_{1}$, $p_{2}$ 
are transformed to four bidegree $(1, 0)$ curves and four bidegree $(0, 1)$ curves on ${\proj}^1\times {\proj}^1$. 
More precisely, the correspondence can be written as 
\begin{equation*}
L_1\to (1)\times {\proj}^1, \quad L_2\to (0)\times {\proj}^1, \quad 
L_3\to (\lambda_1 )\times {\proj}^1, \quad E_{2}\to (\infty )\times {\proj}^1, 
\end{equation*}
\begin{equation*}
L_4\to {\proj}^1\times (1), \quad L_5\to {\proj}^1\times (0), \quad 
L_6\to {\proj}^1\times (\lambda_2 ), \quad E_{1}\to {\proj}^1\times (\infty ), 
\end{equation*}
where $\lambda_1, \lambda_2$ vary in $\mathbb{A}^{1}\backslash \{ 0, 1 \}$. 
%The sum of these curves are the bidegree $(4, 4)$ curves on ${\proj}^1\times {\proj}^1$ considered in \cite{Sa}. 
This correspondence defines an open embedding 
\begin{equation*}
\tilde{U}_{18}/{\rm PGL}_3 \hookrightarrow \mathbb{A}^{1}_{\lambda_1}\times \mathbb{A}^{1}_{\lambda_2}. 
\end{equation*}
The seventh line $L=\overline{q_1 q_2}$ on ${\plane}$ is transformed to 
the (unique) bidegree $(1, 1)$ curve $H_{0}$ on ${\proj}^1\times {\proj}^1$ 
joining the three points $p_{0,0}$, $p_{1,1}$, $p_{\infty,\infty}$, 
where $p_{x,y}$ means the point on ${\proj}^1 \times {\proj}^1$ with inhomogeneous coordinates $(x, y)$. 
Thus, by this correspondence, the sextics parametrized by $\tilde{U}_{18}$ are transformed to 
the above bidegree $(4, 4)$ curves on ${\proj}^1\times {\proj}^1$, 
and the weak marking $(L, q_1, q_2)$ is transformed to $(H_0, p_{1,1}, p_{0,0})$. 
This is the situation considered in \cite{Sa}. 

%This is exactly the construction in \cite{Sa}. 

In \cite{Sa}, after base change by an etale map 
$T\to \mathbb{A}^{1}_{\lambda_1}\times \mathbb{A}^{1}_{\lambda_2}$, 
several families of higher Chow cycles on the $K3$ family over $T$ were constructed. 
The family $\xi$ considered here is essentially one of them: 
more precisely, $\xi_{1}$ in the notation of \cite{Sa} Section 4. 
In \cite{Sa} Theorem 7.1, it is proved that 
the image of $\nu_{tr}((\xi_{1})_{t})$ by the projection $J_{tr}(X)\to H^{2,0}(X)^{\vee}/T(X)^{\vee}$ 
is non-torsion for very general $t\in T$. 
Hence $\nu_{tr}((\xi_{1})_{t})$ is non-torsion for such $t$. 
It follows that $\nu_{-}((\xi_{1})_{t})$ is non-torsion for very general $t\in T$. 
In view of Lemma \ref{lem: regulator conjugate}, 
we see that for a very general point of $\mathbb{A}^{1}_{\lambda_1}\times \mathbb{A}^{1}_{\lambda_2}$, 
the associated cycle has non-torsion anti-invariant regulator for either choice of strong marking. 
Going back to ${\plane}$, 
we see that for a very general point of $\tilde{U}_{18}/{\rm PGL}_{3}$, 
the anti-invariant regulator of our cycle is non-torsion for either choice of strong marking. 
Thus $\nu_{-}(\xi)(u)$ is non-torsion for very general $u\in \hat{U}_{18}$. 
This is the property required in \S \ref{ssec: nontriviality} for starting our induction.

\begin{remark}\label{rmk: irreducible 18}
The explicit description in \cite{Sa} Section 4.2 shows that 
the two conjugate strong markings for the above family of weakly-marked bidegree $(4, 4)$ curves 
cannot be globally distinguished over 
$\mathbb{A}^{1}_{\lambda_1}\times \mathbb{A}^{1}_{\lambda_2}$. 
This means that the two points of a fiber of $\hat{U}_{18}\to \tilde{U}_{18}$ 
can be transformed to each other by the monodromy. 
Therefore $\hat{U}_{18}$ is connected. 
\end{remark}

\begin{remark}
This switch to the ${\proj}^1\times {\proj}^1$ model lasts until $r=5$, 
i.e., as far as the sextic $B$ has two more singular points other than $q_1$, $q_2$. 
The ${\proj}^1\times {\proj}^1$ model itself lasts until $r=4$. 
In this paper we adopt the ${\proj}^2$ model mainly because 
we can directly use the series of families constructed in \cite{Ma1}. %together with the degeneration relation. 
\end{remark}

\subsection{The case $r=17$}\label{ssec: r=17}

We define parameter spaces in the case $r=17$ 
by moving the three lines $L_1$, $L_2$, $L_3$ in the case $r=18$ to general position. 
We first define $\tilde{U}_{17}$ as the codimension $1$ locus in ${\lin}^6$ parametrizing 
six ordered distinct lines $(L_1, \cdots, L_6)$ such that 
$L_4, L_5, L_6$ intersect at one point, 
and no other three of $L_1, \cdots, L_6$ intersect at one point. 
(See Figure 3.) 
Then let $U_{17}\subset {\sextic}$ be the image of the natural map 
\begin{equation*}
\tilde{U}_{17}\to {\sextic}, \qquad (L_1, \cdots, L_6)\mapsto L_1+\cdots +L_6. 
\end{equation*} 
The projection $\tilde{U}_{17}\to U_{17}$ is an $\frak{S}_{3}\times \frak{S}_{3}$-cover 
which forgets labeling of the six lines. 
For $(L_1, \cdots, L_6)\in \tilde{U}_{17}$ we define a weak marking of $L_1+\cdots +L_6$ 
by selecting the nodes as 
\begin{equation}\label{eqn: weak marking 17}
q_1 = L_1\cap L_4, \qquad q_2 = L_2\cap L_5. 
\end{equation}
In \cite{Ma1} \S 9.4, it is proved that 
the 2-elementary $K3$ surfaces associated to the sextics in $U_{17}$ have invariant $(r, a, \delta)=(17, 5, 1)$, 
and $U_{17}$ dominates the moduli space $\mathcal{M}_{17,5,1}$. 

\begin{figure}[h]\label{figure: r17}
\includegraphics[height=45mm, width=65mm]{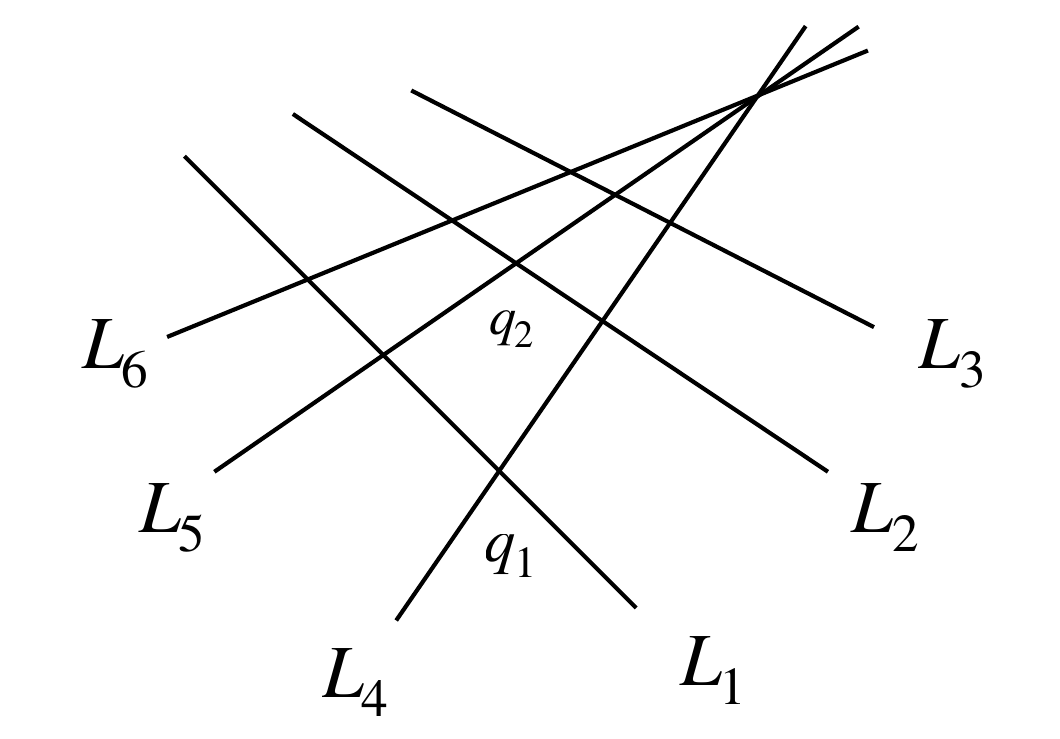}
\caption{$r=17$}
\end{figure}

Next we define partial compactifications by allowing $L_1, L_2, L_3$ to intersect at one point. 
Thus $\tilde{U}_{17}'$ is defined as the locus in ${\lin}^6$ parametrizing six distinct lines $(L_1, \cdots, L_6)$ such that 
$L_4, L_5, L_6$ intersect at one point, 
and no other three of $L_1, \cdots, L_6$ possibly except $L_1, L_2, L_3$ intersect at one point. 
%Then $\tilde{U}_{17}'$ is an open set of the closure of $\tilde{U}_{17}$ in ${\lin}^6$. 
This partial compactification $\tilde{U}_{17}'$ is still smooth. 
Clearly the boundary $\tilde{U}_{17}'-\tilde{U}_{17}$ coincides with $\tilde{U}_{18}$. 
The weak marking \eqref{eqn: weak marking 17} extends over $\tilde{U}_{17}'$, 
and at the boundary $\tilde{U}_{17}'-\tilde{U}_{17}$ this coincides with 
the one \eqref{eqn: weak marking 18} for $\tilde{U}_{18}$. 

Finally, we define $U_{17}'$ as the image of the natural map $\tilde{U}_{17}'\to {\sextic}$. 
Then $U_{17}'-U_{17}=U_{18}$. 
Note that $U_{17}'$ is non-normal at the boundary $U_{18}$. 
It has two branches corresponding to the choice of which triple intersection point to get resolved. 
If we take the normalization of $U_{17}'$, its boundary parametrizes 
\textit{ordered} pairs $(L_1+L_2+L_3, L_4+L_5+L_6)$ of three unordered lines meeting at one point. 
This is an etale double cover of $U_{18}$. 
Then the projection $\tilde{U}_{17}'\to U_{17}'$ factors through this normalization. 
This explains the difference of the degree of $\tilde{U}_{17}\to U_{17}$ and that of $\tilde{U}_{18}\to U_{18}$.

\subsection{The case $r=16$}\label{ssec: r=16}

We define parameter spaces in the case $r=16$ 
by moving the three lines $L_4$, $L_5$, $L_6$ in the case $r=17$ to general position. 
We first define $\tilde{U}_{16}$ as the open locus in ${\lin}^6$ parametrizing 
six ordered distinct lines $(L_1, \cdots, L_6)$ such that 
no three of them intersect at one point. 
(See Figure 4.) 
Then let $U_{16}\subset {\sextic}$ be the image of the natural map 
\begin{equation*}
\tilde{U}_{16}\to {\sextic}, \qquad (L_1, \cdots, L_6)\mapsto L_1+\cdots +L_6. 
\end{equation*}
The projection $\tilde{U}_{16}\to U_{16}$ an $\frak{S}_{6}$-cover which forgets labeling of the six lines. 
For $(L_1, \cdots, L_6)\in \tilde{U}_{16}$ we define a weak marking of $L_1+\cdots +L_6$ by 
\begin{equation}\label{eqn: weak marking 16}
q_1 = L_1\cap L_4, \qquad q_2 = L_2\cap L_5. 
\end{equation}
The 2-elementary $K3$ surfaces associated to the sextics in $U_{16}$ have invariant $(16, 6, 1)$. 
This family of $K3$ surfaces was studied extensively by Matsumoto-Sasaki-Yoshida \cite{MSY}. 
The map $U_{16}\to \mathcal{M}_{16,6,1}$ to the moduli space is dominant (\cite{MSY}, \cite{Ma1}). 

\begin{figure}[h]\label{figure: r16}
\includegraphics[height=45mm, width=65mm]{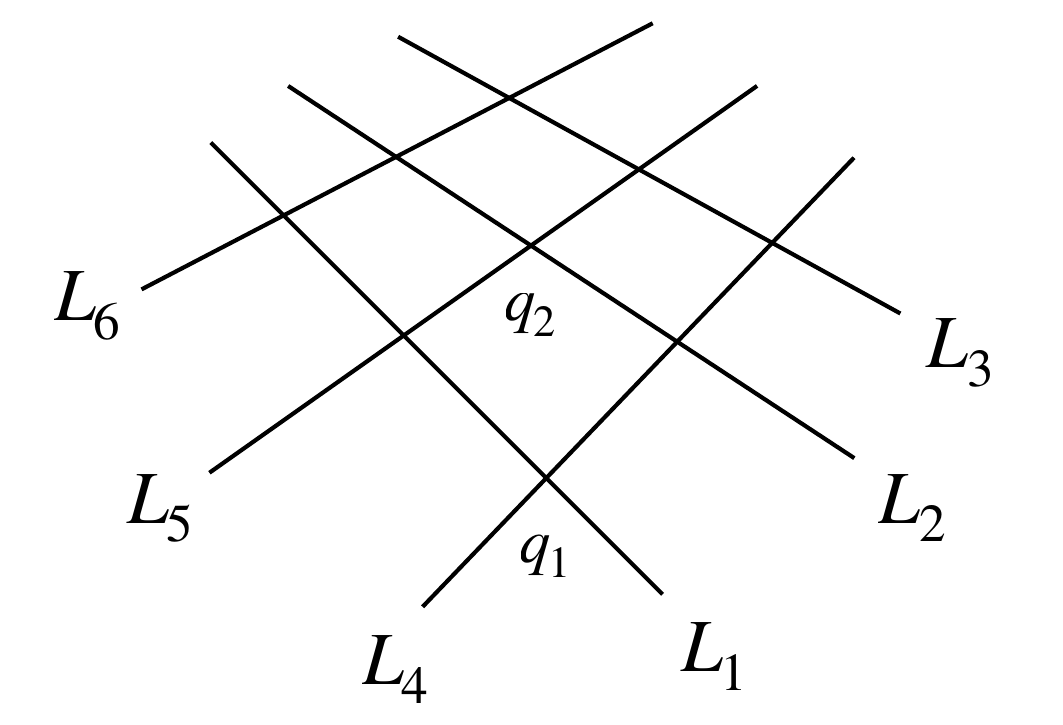}
\caption{$r=16$}
\end{figure}

We define partial compactifications by allowing $L_4, L_5, L_6$ to intersect at one point. 
Thus $\tilde{U}_{16}'$ is defined as the open locus in ${\lin}^6$ 
parametrizing six distinct lines $(L_1, \cdots, L_6)$ such that 
no three of them possibly except $L_4, L_5, L_6$ intersect at one point. 
The boundary $\tilde{U}_{16}'-\tilde{U}_{16}$ coincides with $\tilde{U}_{17}$.  
The weak marking \eqref{eqn: weak marking 16} at the boundary $\tilde{U}_{16}'-\tilde{U}_{16}$ 
coincides with the one \eqref{eqn: weak marking 17} for $\tilde{U}_{17}$. 
Finally, we let $U_{16}'$ be the image of the natural map $\tilde{U}_{16}'\to {\sextic}$. 
Then $U_{16}'-U_{16}=U_{17}$. %, and $U_{16}'$ is still smooth at the boundary. 
The map $\tilde{U}_{16}'\to U_{16}'$ is non-proper over the boundary 
due to the lack of other components of the $\frak{S}_{6}$-orbit of $\tilde{U}_{17}$.

\begin{remark}
If we impose the condition that $L_1, \cdots, L_6$ are tangent to some smooth conic, 
the $K3$ surface $X$ is the Kummer surface of a principally polarized abelian surface. 
The corresponding locus is of codimension $1$ in $U_{15}$. 
Sreekantan \cite{Sr} constructed a series of infinitely many higher Chow cycles on such Kummer surfaces. 
It appears that our cycles restricted on the Kummer locus 
agree with the case $d=1$ of Sreekantan's cycles. 
\end{remark}

\begin{remark}
$K3$ surfaces in $\mathcal{M}_{16,6,1}$ have been studied from various viewpoints. 
For example, elliptic fibrations on very general members were classified in \cite{Kl}. 
The technique of \cite{Kl} was extended in \cite{CG} to other cases $(r, a, \delta)$ with $r+a=22$, 
namely the same specialization/generalization line as ours. 
\end{remark}

\subsection{The case $r=15$}\label{ssec: r=15}

We define parameter spaces in the case $r=15$ 
by smoothing $L_3+L_6$ in the case $r=16$ to smooth conics. 
First we define $\tilde{U}_{15}$ as the open locus in ${\conic}\times {\lin}^4$ parametrizing 
tuples $(Q, L_1, \cdots, L_4)$ such that 
$Q$ is a smooth conic and $Q+L_1+\cdots +L_4$ has at most nodes. 
(See Figure 5.) 
Then we let $U_{15}$ be the image of the natural map 
\begin{equation*}
\tilde{U}_{15}\to {\sextic}, \qquad (Q, L_1, \cdots, L_4)\mapsto Q+L_1+\cdots +L_4. 
\end{equation*}
The projection $\tilde{U}_{15}\to U_{15}$ is an $\frak{S}_{4}$-cover which forgets labeling of the four lines. 
For $(Q, L_1, \cdots, L_4)\in \tilde{U}_{15}$ we define a weak marking of $Q+L_1+\cdots +L_4$ by 
\begin{equation}\label{eqn: weak marking 15}
q_1 = L_1\cap L_3, \qquad q_2 = L_2\cap L_4. 
\end{equation}
The 2-elementary $K3$ surfaces associated to the sextics in $U_{15}$ have invariant $(15, 7, 1)$,  
and $U_{15}$ dominates the moduli space $\mathcal{M}_{15,7,1}$ (\cite{Ma1} \S 9.2).

\begin{figure}[h]\label{figure: r15}
\includegraphics[height=45mm, width=65mm]{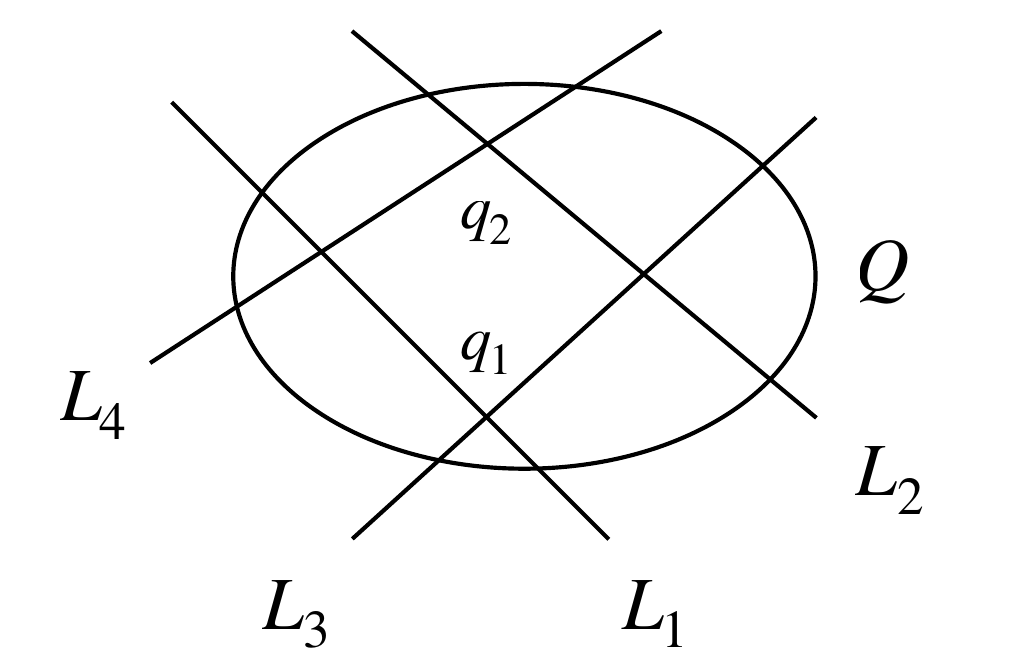}
\caption{$r=15$}
\end{figure}

Next we define partial compactifications by allowing the conic $Q$ to split. 
Thus let $\tilde{U}_{15}'\subset {\conic}\times {\lin}^4$ be the locus of tuples 
$(Q, L_1, \cdots, L_4)$ such that $Q+L_1+\cdots +L_4$ has at most nodes. 
We have the etale map of degree $2$ 
\begin{equation}\label{eqn: boundary r15}
\tilde{U}_{16}\to \tilde{U}_{15}'-\tilde{U}_{15}, \quad 
(L_1, \cdots, L_6)\mapsto (L_3+L_6, L_1, L_2, L_4, L_5). 
\end{equation}
The weak marking \eqref{eqn: weak marking 15} extends over $\tilde{U}_{15}'$, 
and its pullback by \eqref{eqn: boundary r15} agrees with the weak marking \eqref{eqn: weak marking 16} for $\tilde{U}_{16}$. 

Finally, $U_{15}'$ is defined as the image of the natural map $\tilde{U}_{15}'\to {\sextic}$. 
Clearly we have $U_{15}'-U_{15}=U_{16}$. 
Note that $U_{15}'$ is non-normal at the boundary: 
it has $15=\binom{6}{2}$ branches corresponding to the choice of which two lines to be smoothed. 
Hence the normalization of $U_{15}'$ has degree $15$ over the boundary. 
Together with the degree $2$ of \eqref{eqn: boundary r15}, 
this explains the difference of the degree of 
$\tilde{U}_{15}\to U_{15}$ and that of $\tilde{U}_{16}\to U_{16}$.

\subsection{The case $r=14$}\label{ssec: r=14}

We define parameter spaces in the case $r=14$ 
by partially smoothing $Q+L_3$ in the case $r=15$ to irreducible nodal cubics. 
We denote by $V_{nc}\subset{\cubic}$ the codimension $1$ locus of irreducible nodal cubics. 
Let $\tilde{U}_{14}$ be the locus in 
$V_{nc}\times {\lin}^3\times {\plane}$ parametrizing 
tuples $(C, L_1, L_2, L_3, p)$ such that 
$C+L_1+ L_2 +L_3$ has at most nodes and $p\in C\cap L_1$. 
(See Figure 6.) 
Then let $U_{14}$ be the image of the natural map 
\begin{equation*}
\tilde{U}_{14}\to {\sextic}, \qquad (C, L_1, L_2, L_3, p)\mapsto C + L_1 + L_2 + L_3. 
\end{equation*}
The covering $\tilde{U}_{14}\to U_{14}$ endows the sextics $C + L_1 + L_2 + L_3$ 
with labelings of the lines $L_1, L_2, L_3$ and choices of a point $p$ from $C\cap L_1$. 
Hence it has degree $18=3! \cdot 3$. 
For $(C, L_1, L_2, L_3, p)\in \tilde{U}_{14}$ we define a weak marking of $C + L_1 + L_2 + L_3$ by  
\begin{equation}\label{eqn: weak marking 14}
q_1 = p, \qquad q_2 = L_2\cap L_3. 
\end{equation}
The associated 2-elementary $K3$ surfaces have invariant $(14, 8, 1)$,  
and $U_{14}$ dominates the moduli space $\mathcal{M}_{14,8,1}$ (\cite{Ma1} \S 9.1). 
This family of $K3$ surfaces was also studied in \cite{KSTT}. 

\begin{figure}[h]\label{figure: r14}
\includegraphics[height=45mm, width=65mm]{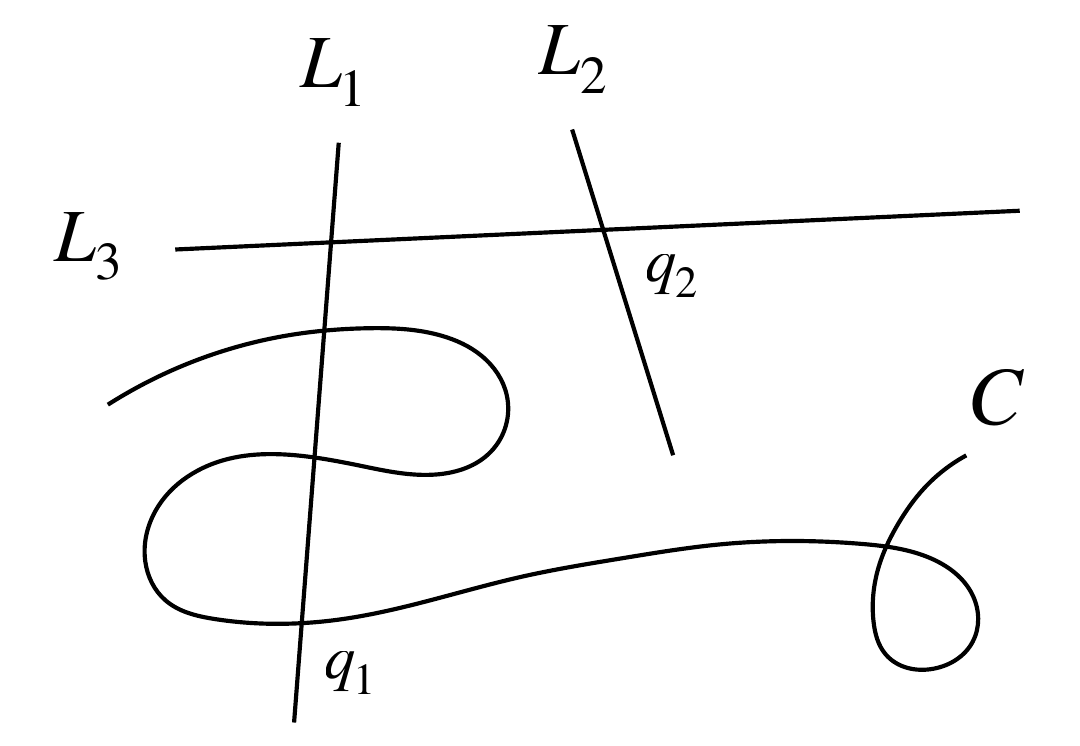}
\caption{$r=14$}
\end{figure}

Next we define partial compactifications by allowing the nodal cubic to be reducible. 
Let $V_{nc}'\subset {\cubic}$ be the locus of nodal cubics 
which is either irreducible or sum of a smooth conic and a line. 
%Then $V_{nc}'$ is an open locus of the closure of $V_{nc}$ in ${\cubic}$. 
%(This is non-normal at the boundary $V_{nc}'-V_{nc}$: see \cite{KS} \S 2.) 
We define $\tilde{U}_{14}'$ as the locus in 
$V_{nc}'\times {\lin}^3\times {\plane}$ parametrizing 
tuples $(C, L_1, L_2, L_3, p)$ such that 
$C+L_1+ L_2 +L_3$ has at most nodes and $p\in C\cap L_1$, 
plus the condition that when $C$ is reducible, $p$ is the intersection of the line component of $C$ with $L_1$. 
%Then $\tilde{U}_{14}'$ is an open set of the closure of $\tilde{U}_{14}$. 
(If we allow $p$ from the intersection of the conic component with $L_1$, we get an additional boundary divisor.) 
The boundary $\tilde{U}_{14}'-\tilde{U}_{14}$ is identified with $\tilde{U}_{15}$ by the isomorphism 
\begin{equation*}
\tilde{U}_{15}\to \tilde{U}_{14}'-\tilde{U}_{14}, \quad 
(Q, L_1, \cdots, L_4)\mapsto (Q+L_3, L_1, L_2, L_4, L_1\cap L_3). 
\end{equation*}
The weak marking \eqref{eqn: weak marking 14} extends over $\tilde{U}_{14}'$. 
At the boundary it agrees with 
the weak marking \eqref{eqn: weak marking 15} for $\tilde{U}_{15}$ via this map. 

Finally, we let $U_{14}'$ be the image of the natural map $\tilde{U}_{14}'\to {\sextic}$. 
We have $U_{14}'-U_{14}=U_{15}$. 
Both $\tilde{U}_{14}'$ and $U_{14}'$ are non-normal at the boundary: 
$U_{14}'$ has eight branches corresponding to the choice of 
which point of $Q\cap L_1+\cdots +L_4$ to get resolved, 
and $\tilde{U}_{14}'$ has two branches corresponding to the choice of 
which point of $Q\cap L_3$ to get resolved. 
(The non-normality of $\tilde{U}_{14}'$ inherits that of $V_{nc}'$.) 
Note that the non-normality of $\tilde{U}_{14}'$ does not affect our construction, 
as the limit weak marking at the boundary does not depend on the choice of branch from which we approach.  
%or equivalently, the choice of the point of $Q\cap L_3$ to be smoothed. 
%Due to our forgotten boundary component, 
%the map $\tilde{U}_{14}'\to U_{14}'$ is non-proper over the boundary. 

\subsection{The case $r=13$}\label{ssec: r=13}

Parameter spaces in the case $r=13$ are defined by smoothing $L_1+L_2$ in the case $r=14$. 
Let $V_{nc}\subset {\cubic}$ be as in \S \ref{ssec: r=14}. 
We define 
\begin{equation*}
\tilde{U}_{13} \: \subset \: V_{nc}\times {\conic} \times {\lin}\times ({\plane})^2 
\end{equation*}
as the locus of tuples $(C, Q, L, q_1, q_2)$ such that 
$Q$ is smooth, $C+Q+L$ has at most nodes, $q_1\in C\cap Q$ and $q_2\in Q\cap L$. 
(See Figure 7.) 
Then let $U_{13}$ be the image of the natural map 
\begin{equation*}
\tilde{U}_{13}\to {\sextic}, \qquad (C, Q, L, q_1, q_2)\mapsto C + Q + L. 
\end{equation*}
The covering $\tilde{U}_{13}\to U_{13}$ 
attaches the weak markings $(q_1, q_2)$ to the sextics $C+Q+L$. 
It has degree $12=6\cdot 2$. 
The associated 2-elementary $K3$ surfaces have invariant $(13, 9, 1)$,  
and $U_{13}$ dominates the moduli space $\mathcal{M}_{13,9,1}$ (see \cite{Ma1} \S 8.1). 

\begin{figure}[h]\label{figure: r13}
\includegraphics[height=48mm, width=62mm]{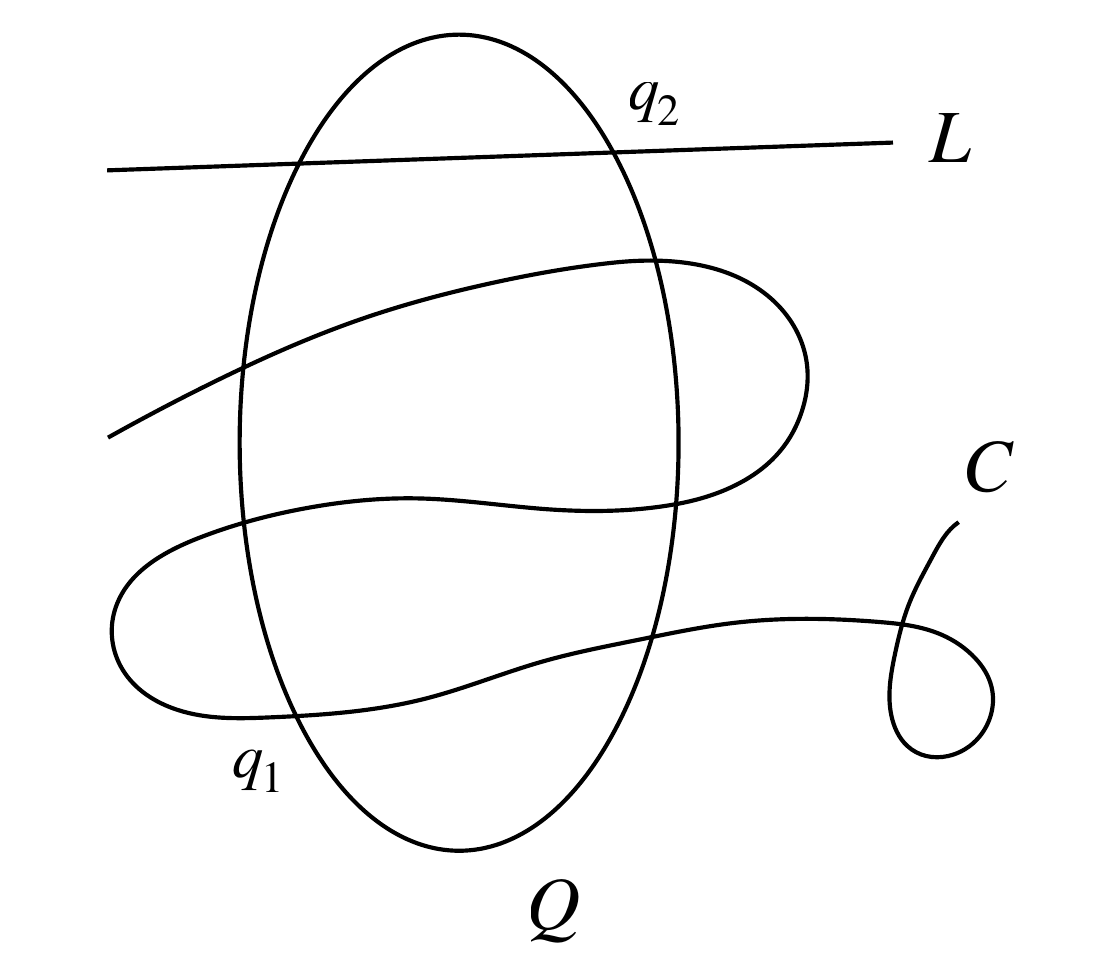}
\caption{$r=13$}
\end{figure}

Partial compactifications are defined by allowing the conic $Q$ to split. 
Thus we define 
\begin{equation*}
\tilde{U}_{13}' \: \subset \: V_{nc}\times {\conic} \times {\lin}\times ({\plane})^2 
\end{equation*}
by the same conditions as for $\tilde{U}_{13}$ except that 
we require $Q$ only to be reduced, and impose the condition that 
when $Q$ splits, $q_1$ and $q_2$ belong to different components of $Q$. 
Note that $\tilde{U}_{13}'$ is still smooth at the boundary $\tilde{U}_{13}'-\tilde{U}_{13}$. 
We have the natural isomorphism 
\begin{equation*}
\tilde{U}_{14}\to \tilde{U}_{13}'-\tilde{U}_{13}, \quad 
(C, L_1, L_2, L_3, p)\mapsto (C, L_1+L_2, L_3, p, L_2\cap L_3). 
\end{equation*}
The limit weak marking $(q_1, q_2)$ at the boundary of $\tilde{U}_{13}'$ coincides with  
the weak marking \eqref{eqn: weak marking 14} for $\tilde{U}_{14}$ via this isomorphism. 

Finally, $U_{13}'$ is defined as the image of the natural map $\tilde{U}_{13}'\to {\sextic}$. 
We have $U_{13}'-U_{13}=U_{14}$. 
Then $U_{13}'$ has three branches at the boundary corresponding to the choice of 
which two of the three lines to be smoothed.

\subsection{The case $r=12$}\label{ssec: r=12}

Parameter spaces in the case $r=12$ are defined 
by partially smoothing $Q+L$ in the case $r=13$. 
Thus we let $\tilde{U}_{12}$ be the locus in 
$(V_{nc})^2\times {\plane}$ parametrizing triplets $(C_1, C_2, p)$ such that 
$C_1+C_2$ has at most nodes and $p\in C_1\cap C_2$. 
(See Figure 8.) 
Then let $U_{12}$ be the image of the natural map 
\begin{equation*}
\tilde{U}_{12}\to {\sextic}, \qquad (C_1, C_2, p)\mapsto C_1 + C_2. 
\end{equation*}
The covering $\tilde{U}_{12}\to U_{12}$ endows the sextics $C_{1}+C_{2}$ with 
labelings of the components $C_1, C_2$ and choices of a point $p$ from $C_1\cap C_2$. 
Hence it has degree $18=2\cdot 9$. 
This marking is equivalent to the weak marking 
\begin{equation}\label{eqn: weak marking 12}
q_1 = p, \qquad q_2 = {\rm Sing}(C_2). 
\end{equation}
The associated 2-elementary $K3$ surfaces have invariant $(12, 10, 1)$,  
and $U_{12}$ dominates the moduli space $\mathcal{M}_{12,10,1}$ (see \cite{Ma1} \S 7.1). 

\begin{figure}[h]\label{figure: r12}
\includegraphics[height=44mm, width=64mm]{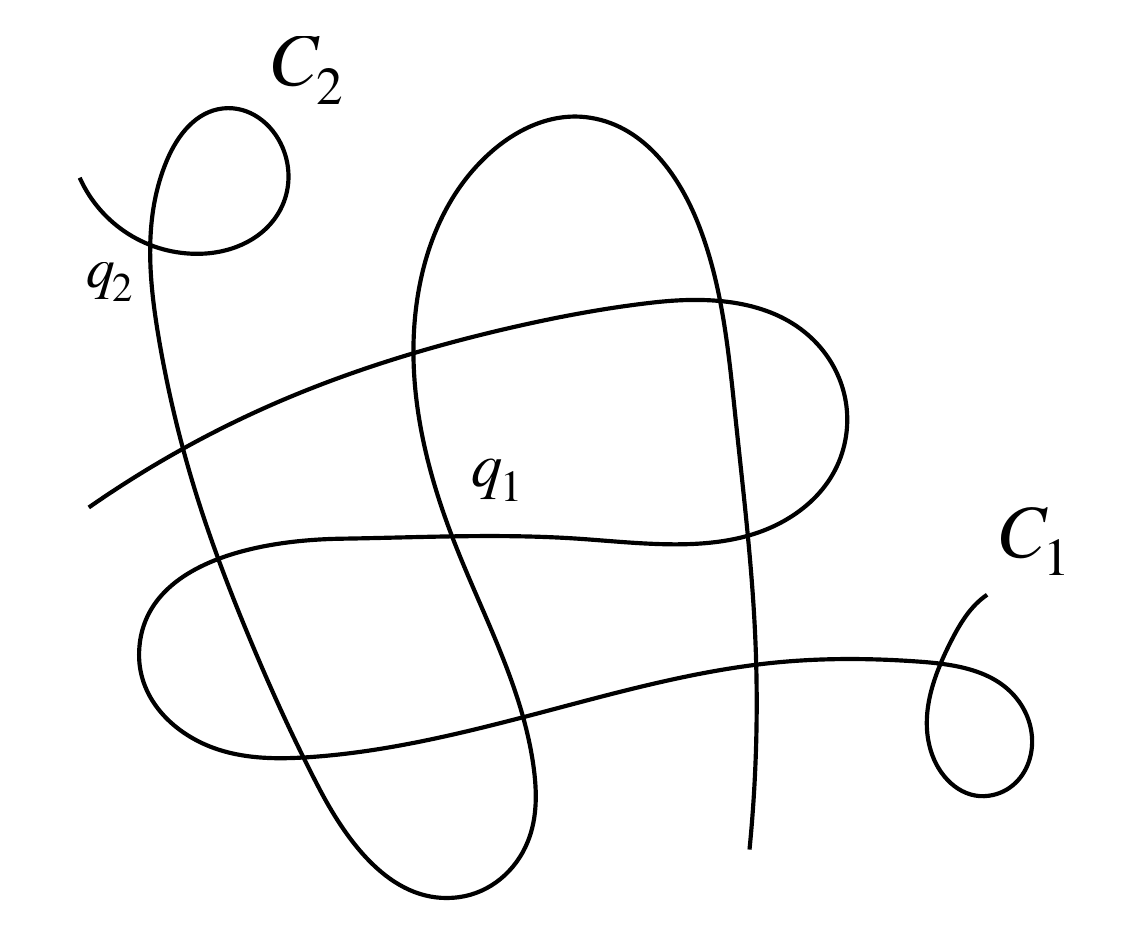}
\caption{$r=12$}
\end{figure}

Partial compactifications can be obtained by allowing $C_2$ to split, 
but here we have to be careful as regards to limit of weak marking. 
Let $V_{nc}'$ be the partial compactification of $V_{nc}$ considered in \S \ref{ssec: r=14} 
whose boundary is the locus of sum of smooth conics $Q$ and lines $L$ meeting transversely. 
Then $V_{nc}'$ is non-normal at the boundary, 
where it has two branches corresponding to the choice of which of $Q\cap L$ to be resolved (see \cite{KS} \S 2). 
If we simply take the closure of $\tilde{U}_{12}$ in $V_{nc}\times V_{nc}'\times {\plane}$, 
then we would get two limit weak markings at the boundary depending on which branch we approach from: 
$q_2$ is the point of $Q\cap L$ which is not resolved. 
Thus the limit weak marking is not well-defined.  
For this reason, we need to take the normalization of $V_{nc}'$. 

So let $V_{nc}''$ be the normalization of $V_{nc}'$. 
At the boundary, this gives an etale double cover of the boundary of $V_{nc}'$ which chooses one of $Q\cap L$. 
By abuse of notation, we denote a boundary point of $V_{nc}''$ still by $C$. 
This means a pair $(Q+L, q)$ where $Q$ is a smooth conic, 
$L$ is a line meeting $Q$ transversely, and $q\in L\cap Q$. 
We take the convention that $q$ stands for the point that is not resolved. 
Now we define 
\begin{equation*}
\tilde{U}_{12}' \; \subset \; V_{nc}\times V_{nc}'' \times {\plane} 
\end{equation*}
as the locus of triplets $(C_1, C_2, p)$ 
such that $C_1+C_2$ has at most nodes, $p\in C_1\cap C_2$, 
and when $C_2$ splits, $p$ is contained in the intersection of $C_1$ with the conic component of $C_2$. 
%This is a Zariski open set of the closure of $\tilde{U}_{12}$ in $V_{nc}\times V_{nc}'' \times {\plane}$. 
In this way, the weak marking \eqref{eqn: weak marking 12} extends to weak marking over $\tilde{U}_{12}'$: 
at the boundary, this is given by $(p, q)$ where $q$ is the point of $Q\cap L$ which is not resolved. 
Then the boundary $\tilde{U}_{12}'-\tilde{U}_{12}$ is identified with $\tilde{U}_{13}$ by the isomorphism 
\begin{equation*}
\tilde{U}_{13}\to \tilde{U}_{12}'-\tilde{U}_{12}, \quad (C, Q, L, q_1, q_2)\mapsto (C, (Q+L, q_2), q_1). 
\end{equation*}
It is now clear that the limit weak marking at the boundary of $\tilde{U}_{12}'$ agrees with  
the weak marking for $\tilde{U}_{13}$ via this map. 

Finally, we let $U_{12}'$ be the image of the natural map $\tilde{U}_{12}'\to {\sextic}$. 
We have $U_{12}'-U_{12}=U_{13}$. 
Then $U_{12}'$ has two branches at the boundary, 
corresponding to the non-normality of $V_{nc}'$.

\subsection{The case $r=11$}\label{ssec: r=11}

Here we arrive at the top of the Nikulin mountain \eqref{eqn: Nikulin mountain}. 
This is the case where we encounter with Coble curves, 
namely irreducible ten-nodal rational plane sextics. 
Let $U_{11}\subset {\sextic}$ be the locus of Coble curves. 
We define $\tilde{U}_{11}$ by attaching weak markings to Coble curves. 
Thus $\tilde{U}_{11}$ is the locus in $U_{11}\times ({\plane})^2$ parametrizing triplets 
$(C, q_1, q_2)$ such that $q_1\ne q_2$ are nodes of $C$. 
The map $\tilde{U}_{11}\to U_{11}$ has degree $90=10\cdot 9$. 
The associated 2-elementary $K3$ surfaces have invariant $(11, 11, 1)$,  
and $U_{11}$ dominates the moduli space $\mathcal{M}_{11,11,1}$ (\cite{Ma1} \S 6.1). 

The closure of $U_{11}$ in ${\sextic}$ contains $U_{12}$, 
the locus of sum $C_1+C_2$ of two irreducible nodal cubics. 
This is a Zariski open set of one of the irreducible components of the boundary. 
We let $U_{11}'=U_{11}\cup U_{12}$. 
Note that this is non-normal at the boundary. 
It has nine branches corresponding to the choice of which point of $C_1\cap C_2$ to be resolved. 
(The nodes of $C_1$ and $C_2$ themselves will not be resolved when deforming $C_1+C_2$ to Coble curves.) 

We define the partial compactification $\tilde{U}_{11}'$ of $\tilde{U}_{11}$ as follows. 
The closure of $\tilde{U}_{11}$ in $U_{11}'\times ({\plane})^2$ has 
four boundary divisors over $U_{12}$ corresponding to the following four types 
of configuration of limit $(q_1, q_2)$: 
\begin{enumerate}
\item[(A)] $q_1, q_2\in C_1\cap C_2$, 
\item[(B)] $q_1\in C_1\cap C_2$, $q_2={\rm Sing}(C_2)$, 
\item[(C)] $q_1={\rm Sing}(C_1)$, $q_2\in C_1\cap C_2$, 
\item[(D)] $q_1={\rm Sing}(C_1)$, $q_2={\rm Sing}(C_2)$.  
\end{enumerate}
A Zariski open set of the boundary divisor of type (B) is identified with $\tilde{U}_{12}$ by 
setting $p=q_1$ and $C_2$ to be the component which contains $q_2$.  
Then we let $\tilde{U}_{11}'$ be the union of $\tilde{U}_{11}$ with this boundary locus: 
\begin{equation*}
\tilde{U}_{11}' = \tilde{U}_{11} \cup \tilde{U}_{12}. 
\end{equation*}
Clearly the limit weak marking at the boundary of $\tilde{U}_{11}'$ coincides with 
the weak marking \eqref{eqn: weak marking 12} for $\tilde{U}_{12}$. 

Note that the closure of $\tilde{U}_{11}$ in $U_{11}'\times({\plane})^2$ is non-normal at every boundary divisor. 
It has $7$, $8$, $8$, $9$ branches at the boundary divisors of type (A), (B), (C), (D) respectively, 
corresponding to the choice of an \textit{unmarked} intersection point of $C_1$ and $C_2$ to be resolved. 
The boundary divisors of type (A), (B), (C), (D) have degrees 
$72$, $18$, $18$, $2$ over $U_{12}$ respectively. 
By considering the map from the normalization of the closure of $\tilde{U}_{11}$ to that of $U_{11}'$, 
we can understand that these numbers are compatible: 
\begin{equation*}
9 \cdot 90 = 7\cdot 72 + 8 \cdot 18 + 8 \cdot 18 + 9 \cdot 2. 
\end{equation*}

\subsection{The case $3\leq r \leq 10$}\label{ssec: 3r10}

Here we go down the left roof of the Nikulin mountain \eqref{eqn: Nikulin mountain}. 
Let $3\leq r \leq 10$. 
Let $U_{r}\subset {\sextic}$ be the locus of irreducible $(r-1)$-nodal sextics. 
Then $U_{r}$ is smooth and irreducible (\cite{Ha}). 
We define the covering $\tilde{U}_{r}\to U_{r}$ by attaching weak markings. 
Thus $\tilde{U}_{r}$ is defined as the locus in $U_{r}\times ({\plane})^{2}$ 
parametrizing triplets $(C, q_1,  q_2)$ such that  $q_1\ne q_2$ are nodes of $C\in U_{r}$. 
The projection $\tilde{U}_{r}\to U_{r}$ has degree $(r-1)(r-2)$. 
The 2-elementary $K3$ surfaces associated to the sextics in $U_r$ have invariant $(r, r, 1)$, 
and $U_r$ dominates the moduli space $\mathcal{M}_{r,r,1}$ (see \cite{Ma1} \S 4.1 and \S 5.3). 
These $K3$ surfaces were first studied in \cite{M-S}. 

The closure of $U_{r}$ in ${\sextic}$ contains $U_{r+1}$ as a Zariski open set of 
one of the irreducible components of the boundary (see \cite{Ha}). 
We set 
\begin{equation*}
U_{r}' = U_r \cup U_{r+1}  \; \; \subset \; {\sextic}. 
\end{equation*}
Then we define $\tilde{U}_{r}'$ to be the closure of $\tilde{U}_{r}$ in $U_{r}'\times ({\plane})^{2}$. 
Its boundary parametrizes curves in $U_{r+1}$ with two nodes labelled (limit of labelled nodes of $C\in U_{r}$). 
This is exactly $\tilde{U}_{r+1}$. 
Thus 
\begin{equation*}
\tilde{U}_{r}' = \tilde{U}_{r}\cup \tilde{U}_{r+1}  \; \; \subset \; U_r'\times ({\plane})^2. 
\end{equation*}
We have a direct inductive structure here. 

Note that both $U_{r}'$ and $\tilde{U}_{r}'$ are non-normal at the boundary: 
$U_{r}'$ has $r$ branches corresponding to the choice of 
which node of $C\in U_{r+1}$ to be resolved; 
$\tilde{U}_{r}'$ has $r-2$ branches corresponding to the choice of 
which \textit{unmarked} node to be resolved. 
The difference of the degree $(r-1)(r-2)$ of $\tilde{U}_{r} \to U_{r}$ 
and the degree $r(r-1)$ of $\tilde{U}_{r+1} \to U_{r+1}$ can be understood by 
considering the map from the normalization of $\tilde{U}_{r}'$ to that of $U_{r}'$.

%\begin{remark}
%We could also define $\tilde{U}_{r}$ by labeling \textit{all} nodes, instead of just two nodes. 
%In this way we obtain an $\frak{S}_{r-1}$-cover of $U_{r}$. 
%Then $\tilde{U}_{r}'$ is defined by taking the closure in $U_{r}''\times ({\plane})^{r-1}$ 
%where $U_{r}''$ is the normalization of $U_{r}'$. 
%We still have $\tilde{U}_{r}' = \tilde{U}_{r}\cup \tilde{U}_{r+1}$ in this approach. 
%One virtue here is that $\tilde{U}_{r}'$ is normal. 
%We need to replace the parameter space for $r=11$, but this does not affect our induction. 
%\end{remark} 

\begin{remark}
The quotient $Y=X/\iota$ is the blow-up of ${\plane}$ at the nodes of $C$. 
If $C\in U_{r}$ is general, $Y$ is a del Pezzo surface of degree $10-r$ when $r\leq 9$, 
and is known as a Halphen surface when $r=10$. 
The $(-1)$-curves over the nodes of $C$ and the strict transform of $L=\overline{q_1 q_2}$ are lines on $Y$. 
Thus our higher Chow cycles can be seen as obtained from two intersecting lines on del Pezzo and Halphen surfaces. 
%In this case, fixing $Y$, 
%we can retake the parameter space as a Zariski open set of the linear system $|\! - \! 2K_{Y}|$.  
%This would be more explicit. 
 \end{remark}

%Next we define the partial compactification of $\tilde{U}_{r}$. 
%In order for the limit marking to be well-defined, 
%we need to take the normalization of $U_{r}'$ (compare with \S \ref{ssec: r=12}). 
%Let $U_{r}''$ be the normalization of $U_{r}'$. 
%At the boundary, this takes an etale cover of $U_{r+1}$ of degree $r$ 
%which parametrizes pairs $(C, q)$ such that 
%$C\in U_{r+1}$ and $q\in {\rm Sing}(C)$ (which is the point to be smoothed). 
%Then we define 
%\begin{equation*}
%\tilde{U}_{r}' \; \; \; \subset \; U_{r}''\times ({\plane})^{r-1} 
%\end{equation*}
%as the closure of $\tilde{U}_{r}$ in $U_{r}''\times ({\plane})^{r-1}$. 
%Its boundary parametrizes tuples $((C, q), q_1, \cdots, q_{r-1})$ 
%where $(C, q)$ is a boundary point of $U_{r}''$ and 
%$q_1, \cdots, q_{r-1}$ are the nodes of $C$ that do not get smoothed, namely 
%${\rm Sing}(C)=\{ q_1, \cdots, q_{r-1}, q \}$. 
%We have the natural identification 
%\begin{equation*}
%\tilde{U}_{r+1} \to \tilde{U}_{r}'-\tilde{U}_{r}, \qquad 
%(C, q_1, \cdots, q_{r}) \mapsto ((C, q_{r}), q_1, \cdots, q_{r-1}). 
%\end{equation*}
%The weak marking for $\tilde{U}_{r}$ extends over $\tilde{U}_{r}'$, 
%and at the boundary this clearly coincides with that for $\tilde{U}_{r+1}$. 
%Thus, when $3\leq r \leq 11$, we have a rather direct inductive structure. 

%%%%%%% Reference %%%%%%%%%%%%%%%%%%%%%%%%%%%%%

\end{document}